\documentclass[a4paper,twoside,11pt,oneside]{article}
\usepackage{verbatim}
\usepackage{graphicx}
\usepackage{listings}
\usepackage{amsmath,amssymb,amsthm}
\usepackage[utf8]{inputenc}
\usepackage[english]{babel}
\usepackage[top=4cm, bottom=4cm, left=2.5cm, right=2.5cm]{geometry}
\usepackage{fullpage}
\usepackage{colonequals}
\usepackage[colorlinks=true,linkcolor=blue,urlcolor= blue]{hyperref}
\usepackage{graphicx}
\usepackage{color}
\usepackage{indentfirst}
\usepackage[OT2,T1]{fontenc}
\usepackage{comment} 
\usepackage{shadethm}
\usepackage{enumerate}
\usepackage{soul}
\usepackage{mathtools}
\usepackage{tikz}
\usetikzlibrary{matrix}
\usepackage[nottoc, notlof, notlot]{tocbibind}

\DeclareMathOperator{\appr}{app}

\DeclareMathOperator{\In}{in}

\DeclareMathOperator{\surf}{surf}
\DeclareMathOperator{\when}{when}
\DeclareMathOperator{\dX}{dX}
\DeclareMathOperator{\dz}{dz}
\DeclareMathOperator{\loc}{loc}
\DeclareMathOperator{\Or}{or}

\title{Rigorous derivation from the water waves equations of some full dispersion shallow water models}
\author{Louis Emerald}

\newtheorem{mydef}{Definition}[section]
\newtheorem{mypp}[mydef]{Proposition}

\newtheorem{mylem}[mydef]{Lemma}

\newtheorem{myrem}[mydef]{Remark}

\newtheorem{mynots}[mydef]{Notations}
\newtheorem{mynot}[mydef]{Notation}

\newtheorem{myhyp}[mydef]{Hypotheses}

\usepackage{microtype}
\begin{document}
\numberwithin{equation}{section}

\maketitle
\begin{abstract}
    In order to improve the frequency dispersion effects of irrotational shallow water models in coastal oceanography, several full dispersion versions of classical models were formally derived in the literature. The idea, coming from G. Whitham in \cite{Whitham67}, was to modify them so that their dispersion relation is the same as the water waves equations. In this paper we construct new shallow water approximations of the velocity potential then deducing ones on the vertically averaged horizontal component of the velocity. We make use of them to derive rigorously from the water waves equations two new Hamiltonian full dispersion models. This provides for the first time non-trivial precision results characterizing the order of approximation of the full dispersion models. They are non-trivial in the sense that they are better than the ones for the corresponding classical models.
\end{abstract}





\section{Introduction}
\subsection{Motivations}

In this work, we consider full dispersion models for the propagation of surface waves in coastal oceanography. It is a class of irrotational shallow water models which have the particularity of having the same dispersion relation as the water waves equations. The first nonlinear full dispersion model appearing in the literature was introduced formally by Whitham in \cite{Whitham67, Whitham74}. It is a modification of the Korteweg–de Vries (KdV) equations called the Whitham equations, see \cite{KleinLinaresPilodEtAl18} for a rigorous comparison between those two equations. The goal here was to describe wave breaking phenomena \cite{Hur17} and Stokes waves of extreme amplitude \cite{EhrnstromWahlen19}. Later on, the same kind of formal modifications has been made on other standard shallow water models such as the Boussinesq or Green-Naghdi systems, thus creating a whole class of full dispersion models. The motivation was to widen the range of validity of the shallow water models, see section 5.3 of \cite{WWP}, and to study the propagation of waves above obstacles, a situation where there is creation of high harmonics which are then freely released, see \cite{BattjesBeji,Dingemans94,WWP}. 

The models obtained by modifying the Boussinesq system, generally called Whitham-Boussinesq systems, have been the subject of active research, see \cite{Carter18, DinvayDutykhKalisch19} for comparative studies, \cite{Dinvay19, Dinvay20, DinvaySelbergTesfahun19, KalischPilod18} for the well-posedness theory, \cite{DinvayNilsson19, NilssonWang19} for some works on solitary waves solutions, and \cite{HurPandey16} for a study on modulational instability (this list is not exhaustive, see also \cite{Pandey19, VargasMaganaPanayotaros16}). In the case of the modified Green-Nagdhi sytems, see \cite{DucheneIsrawiTalhouk16} for a fully justified two-layer one.

However at the best of the author's knowledge no direct derivation of these models from the water waves equations has been done. In this paper we provide asymptotic approximations of the Dirichlet-to-Neumann operator. Then we use them to derive two different Hamiltonian full dispersion systems (see \eqref{FirstFullDispersionModel(InPsi)} and \eqref{SecondFullDispersionModel(InVbar)}) and justify them in the sense of consistency (see definition \ref{Definition consistency}) of the water waves equations with these two models. Subsequently we deduce from them an improved precision result with respect to the one already known for the different full dispersion models appearing in the literature. 

\subsection{Consistency problem}
\noindent
Throughout this paper $d$ will be the dimension of the horizontal variable (denoted $X\in \mathbb{R}^d$).

The starting point of this study is the adimensional water waves problem, that is $(d=1,2)$
\begin{align}\label{WaterWavesEquations}
    \begin{cases}
        \partial_t \zeta - \frac{1}{\mu}\mathcal{G}^{\mu}\psi = 0, \\
        \partial_t \psi + \zeta + \frac{\epsilon}{2}|\nabla\psi|^2-\frac{\epsilon}{\mu}\frac{(\mathcal{G}^{\mu}\psi + \epsilon\mu\nabla\zeta\cdot\nabla\psi)^2}{2(1+\epsilon^2\mu|\nabla\zeta|^2)} = 0.
    \end{cases}
\end{align}
Here
\begin{itemize}
    \item $\nabla$ is the horizontal gradient, i.e.
    \begin{align*}
        \nabla \colonequals \begin{cases} \partial_x, \ \ \when \ \ d=1, \\ (\partial_x, \partial_y)^T, \ \ \when \ \ d=2. \end{cases}
    \end{align*}
    \item The free surface elevation is the graph of $\zeta$, which is a function of time $t$ and horizontal space $X\in\mathbb{R}^d$.
    \item $\psi(t,X)$ is the trace at the surface of the velocity potential. 
    \item $\mathcal{G}^{\mu}$ is the Dirichlet-to-Neumann operator defined later in definition \ref{Definitions Fondamentales}.
\end{itemize}
Moreover every variables and functions in \eqref{WaterWavesEquations} are compared with physical characteristic parameters of same dimension. Among those are the characteristic water depth $H_0$, the characteristic wave amplitude $a_{\surf}$ and the characteristic wavelength $L_x$. From these comparisons appear two adimensional parameters of main importance:
\begin{itemize}
    \item $\mu \colonequals \frac{H_0^{2}}{L_x^{2}}$: the shallow water parameter,
    \item $\epsilon \colonequals \frac{a_{\surf}}{H_0}$: the nonlinearity parameter.
\end{itemize}
We refer to \cite{WWP} for details on the derivation of these equations.
\indent 


Before giving the main definitions of this section, here are two assumptions maintained throughout this paper.
\begin{myhyp}
\begin{itemize} 
\item A fundamental hypothesis in this study will be the lower boundedness by a positive constant of the water depth (non-cavitation assumption):
\begin{align}\label{FundamentalHypothesis}
    \exists h_{\min}>0, \forall X\in \mathbb{R}^d, \ \ h \colonequals 1 + \epsilon \zeta(t,X) \geq h_{\min}.
\end{align}
\item We suppose that the bottom of the sea is flat. The water domain is then defined by $\Omega_t \colonequals  \{ (X,z)\in \mathbb{R}^{d+1}: -1 < z < \epsilon \zeta(X) \}$. 
\end{itemize}
\end{myhyp}




In what follows we need some notations on the functional setting of this paper.
\begin{mynots}\label{functional framework}
    \begin{itemize}
        \item For any $s \geq 0$ we will denote $H^s(\mathbb{R}^d)$ the Sobolev space of order $s$ in $L^2(\mathbb{R}^d)$.
        \item For any $s \geq 1$ we will denote $\dot{H}^s(\mathbb{R}^d) \colonequals \{ f \in L_{\loc}^{2}(\mathbb{R}^d), \ \ \nabla f \in H^{s-1}(\mathbb{R}^d) \}$ the Beppo-Levi space of order $s$.
        \item The $L^2(\mathbb{R}^d)$ norm will be written $|\cdot|_2$. The $L^2(\mathcal{S})$ norm, where $\mathcal{S}\colonequals \mathbb{R}^d \times(-1,0)$ (see definition \ref{TrivialDiffeomorphism}), will be denoted $||\cdot||_2$. 
        \item Denoting $\Lambda^s \colonequals (1-\Delta)^{s/2}$, where $\Delta$ is the Laplace operator in $\mathbb{R}^d$, the $H^s(\mathbb{R}^d)$ norm will be $|\cdot|_{H^s} \colonequals |\Lambda^{s}\cdot|_2$.
    \end{itemize}
\end{mynots}

It is easier to work in a time independant water domain. For that reason, by the mean of a diffeomorphism defined in the next definition, we will straighten our problem.

\begin{mydef}\label{TrivialDiffeomorphism}
    Let $\zeta \in H^{t_0 +1}(\mathbb{R}^d)$ $(t_0 > d/2)$ such that \eqref{FundamentalHypothesis} is satisfied. We define the time-dependant trivial diffeomorphism mapping the flat strip $\mathcal{S} \colonequals \mathbb{R}^d \times (-1,0)$ onto the water domain $\Omega_t$
    \begin{align}\label{TrivialDiffeomorphismEq}
    \begin{array}{ccccc}
    \Sigma_t & \colon & \mathcal{S} \colonequals \mathbb{R}^d\times (-1,0) & \to & \Omega_t \colonequals \{ (X,z)\in \mathbb{R}^{d+1}: -1 < z < \epsilon \zeta(X) \} \\
    & & (X,z) & \mapsto & (X, z + \epsilon \zeta (z+1)).\\
    \end{array}
\end{align}
\end{mydef}

\indent 

We can now define the Dirichlet-to-Neumann operator $\mathcal{G}$ in the flat strip $\mathcal{S}$, see the chapters 2 and 3 in \cite{WWP}.

\begin{mydef}\label{Definitions Fondamentales}
    Let $s \geq 0, t_0 > d/2$, $\psi \in \dot{{H}}^{s+3/2}(\mathbb{R}^d)$ and $\zeta \in H^{t_0+1}(\mathbb{R}^d)$ be such that \eqref{FundamentalHypothesis} is satisfied. Using the trivial diffeomorphism \eqref{TrivialDiffeomorphismEq} we introduce the potential velocity $\phi$ in the flat strip $\mathcal{S}$ through the following variable coefficients elliptic equation
    \begin{align}\label{StraightenedLaplaceProblem}
        \begin{cases} \nabla^{\mu}\cdot P(\Sigma_t) \nabla^{\mu} \phi = 0 \ \ \In \ \ \mathcal{S}, \\
        \phi|_{z=0} = \psi, \ \ \partial_z\phi|_{z=-1}=0, \end{cases}
    \end{align}
    where $\nabla^{\mu}$ is the $(d+1)$-gradient operator defined by $\nabla^{\mu} = (\sqrt{\mu}\nabla^T,\partial_z)^T$,\\ and $P(\Sigma_t) =  \begin{pmatrix} (1+\epsilon \zeta)I_d & -\sqrt{\mu}\epsilon(z+1)\nabla\zeta \\ -\sqrt{\mu}\epsilon(z+1)\nabla\zeta^{T} & \frac{1 + \mu \epsilon^2 (z+1)^2|\nabla\zeta|^2}{1+\epsilon\zeta} \end{pmatrix}$.\\
    Let's denote by $h$ the water depth, $h=1+\epsilon\zeta$. We define the vertically averaged horizontal velocity $\overline{V}[\epsilon\zeta]\psi$ (denoted $\overline{V}$ when no confusion is possible) by the formula 
    \begin{align}\label{VerticallyAveragedHorizontalVelocity}
        \overline{V} = \frac{1}{h} \int_{-1}^{0} [h\nabla\phi -\epsilon(z+1)\nabla \zeta\partial_z\phi] \dz.
    \end{align}
    The Dirichlet-to-Neumann operator $\mathcal{G}^{\mu}[\epsilon\zeta]$ (denoted $\mathcal{G}^{\mu}$ when no confusion is possible) is then defined as
    \begin{align}\label{Dirichlet-to-NeumannOperator}
    \begin{array}{ccccc}
    \mathcal{G}^{\mu} & \colon & \dot{H}^{s+3/2}(\mathbb{R}^d) & \to & H^{s+1/2}(\mathbb{R}^d)\\
    & & \psi & \mapsto & -\mu\nabla\cdot (h\overline{V}).\\
    \end{array}
    \end{align}
\end{mydef}

\indent 

Before stating the first result of this paper, we recall the definition of a Fourier multiplier.
\begin{mydef}
    Let $u:\mathbb{R}^d \to \mathbb{R}^d$ be a tempered distribution, let $\widehat{u}$ be its Fourier transform. Let $F:\mathbb{R}^d\to\mathbb{R}$ be a smooth function with polynomial decay. Then the Fourier multiplier associated with $F(\xi)$ is denoted $\mathrm{F}(D)$ (denoted $\mathrm{F}$ when no confusion is possible) and defined by the formula:
    \begin{align*}
        \widehat{\mathrm{F}(D)u}(\xi) = F(\xi)\widehat{u}(\xi).
    \end{align*}
\end{mydef}

We need also other notations. 
\begin{mynots}
    All the results of this paper will use the following notations, where $C(\circ)$ means a constant depending on $\circ$.\\
    Let $t_0 > d/2$, $s\geq 0$, and $\mu_{\max} > 0$. Given sufficiently regular $\zeta$ and $\psi$ satisfying hypothesis \eqref{FundamentalHypothesis} we will write
    \begin{itemize}
        \item $M_0 \colonequals C(\frac{1}{h_{\min}},\mu_{\max},|\zeta|_{H^{t_0}})$.
        \item $M(s) \colonequals C(M_0,|\zeta|_{H^{\max(t_0+1,s)}})$.
        \item $M \colonequals C(M_0,|\zeta|_{H^{t_0+2}})$.
        \item $N(s) \colonequals C(M(s),|\nabla\psi|_{H^s})$.
    \end{itemize}
\end{mynots}
\begin{myrem}\label{remark on t0}
In this paper, the notation $t_0$ is for a real number larger than $d/2$. However, it is not to be taken too large, we can consider $d/2 < t_0 \leq 2$. So that, when $s\geq 3$, $M(s)$ is in fact $M(s) \colonequals C(\frac{1}{h_{\min}},\mu_{\max},|\zeta|_{H^{s}})$. 
\end{myrem}
The first result of this paper provides asymptotic expansions of the vertically averaged horizontal velocity (which implies ones of the Dirichlet-to-Neumann operator) at order $O(\mu\epsilon)$ or $O(\mu^2\epsilon)$ with estimations of error. It also provides an approximation of the velocity potential at the surface expressed in terms of the vertically averaged horizontal velocity at order $O(\mu^2\epsilon)$.
\begin{mypp}\label{Proposition Approximation Vbar estimates}
    Let $s \geq 0$, and $\zeta \in H^{s+4}(\mathbb{R}^d)$ be such that \eqref{FundamentalHypothesis} is satisfied. Let $\psi \in \dot{H}^{s+5}(\mathbb{R}^d)$, and $\overline{V}$ be as in \eqref{VerticallyAveragedHorizontalVelocity}. Let also $\mathrm{F_1} \colonequals \frac{\tanh{(\sqrt{\mu}|D|)}}{\sqrt{\mu}|D|}$, $\mathrm{F_2} = \frac{3}{\mu|D|}(1- \frac{\tanh{(\sqrt{\mu}|D|)}}{\sqrt{\mu}|D|})$, and $\mathrm{F_3} = \mathrm{F_2}\circ\mathrm{F_1}^{-1}$ be three Fourier multipliers.\\
    The following estimates hold:
    \begin{align}\label{AsymptoticDevelopmentVbar}
        \begin{cases}
            |\overline{V}-\mathrm{F}_1\nabla\psi|_{H^{s}} \leq \mu\epsilon M(s+3) |\nabla\psi|_{H^{s+2}},\\
            |\overline{V}- \mathrm{F}_1\nabla\psi - \frac{\mu \epsilon}{3}[h \nabla\zeta \Delta\psi + \nabla(\zeta(1+h)\Delta\psi)]|_{H^s} \leq \mu^2\epsilon M(s+3)|\nabla\psi|_{H^{s+4}}, \\
            |\overline{V}-\nabla\psi - \frac{\mu}{3h}\nabla(h^3 \mathrm{F}_2 \Delta\psi)|_{H^s} \leq \mu^2\epsilon M(s+3)|\nabla\psi|_{H^{s+4}},\\
            |\overline{V}-\nabla\psi -\frac{\mu}{3h}\nabla(h^3 \mathrm{F}_3 \nabla \cdot \overline{V})|_{H^s} \leq \mu^2\epsilon M(s+4) |\nabla\psi|_{H^{s+4}}.
        \end{cases}
    \end{align}
\end{mypp}

\begin{myrem}
   \begin{itemize}
        \item From \eqref{Dirichlet-to-NeumannOperator} we straightforwardly deduce corresponding estimates for the Dirichlet-to-Neumann operator which we omit to write down since we do not use them in our analysis.
        \item In fact we obtain estimates on the straightened velocity potential inside the fluid which would allow us to reconstruct the velocity field at precision $O(\mu^2\epsilon)$. Let $t_0 > d/2$, and $\phi$ be the solution of \eqref{StraightenedLaplaceProblem}. Let also $\mathrm{F}_0 \colonequals \frac{\cosh{((z+1)\sqrt{\mu}|D|)}}{\cosh{(\sqrt{\mu}|D|)}}$ be a Fourier multiplier depending on the transversal variable $z\in [-1,0]$. Then one has
        \begin{align*}
            \begin{cases}
                ||\Lambda^s\nabla^{\mu}(\phi-\mathrm{F}_0\psi - \mu\epsilon\zeta(1+h)(\frac{z^2}{2}+z)\Delta\psi) ||_2 \leq \mu^2\epsilon M(s+2)|\nabla\psi|_{H^{s+3}},\\
                ||\Lambda^s\nabla^{\mu}(\phi-(\psi+h^2(\mathrm{F}_0-1)\psi) ||_2 \leq \mu^2\epsilon M(s+2) |\nabla\psi|_{H^{s+3}}.
            \end{cases}
        \end{align*}
   \end{itemize}
\end{myrem}

\indent

To state the second result of this paper we need to define the notion of consistency of the water waves equations \eqref{WaterWavesEquations} with respect to a given model in the shallow water asymptotic regime at a certain order in $\mu$ and $\epsilon$.

\begin{mydef}\label{Definition consistency}(Consistency) \\
    Let $\mu_{\max} > 0$. Let $\mathcal{A} \subset \{(\epsilon,\mu), \ \ 0 \leq \epsilon \leq 1, \ \ 0\leq \mu\leq \mu_{\max} \}$ be the shallow water asymptotic regime. We denote by $(A)$ and $(A')$ two asymptotic models of the following form:
    \begin{align*}
        (A) \ \ \begin{cases}
                    \partial_t \zeta + \mathcal{N}_{(A)}^1(\zeta,\psi) = 0, \\
                    \partial_t \psi + \mathcal{N}_{(A)}^2(\zeta,\psi) = 0,
                \end{cases}
                , \ \ 
        (A') \ \ \begin{cases}
                    \partial_t \zeta + \nabla\cdot(h\overline{V}) = 0,\\
                    \partial_t ((Id + \mu T_{(A')}[h]) \overline{V}) + \mathcal{N}^3_{(A')}(\zeta,\overline{V}) = 0,
                 \end{cases}
    \end{align*}
    where $\mathcal{N}_{(A)}^1$, $\mathcal{N}_{(A)}^2$ and $\mathcal{N}^3_{(A')}$ are nonlinear operators that depend respectively on the asymptotic model $(A), (A)$ and $(A')$. And $T_{(A')}$ is an operator nonlinear in $h$ and linear in $\overline{V}$ which depends on the asymptotic model $(A')$. \\
    We say that the water waves equations are consistent at order $O(\mu^k\epsilon^l)$ with respectively $(A)$ or $(A')$ in the regime $\mathcal{A}$ if there exists $n \in \mathbb{N}$ and $T >0$ such that for all $s\geq0$ and $p=(\epsilon,\mu)\in\mathcal{A}$, and for every solution $(\zeta,\psi) \in C([0,\frac{T}{\epsilon}];H^{s+n}\times \dot{H}^{s+n+1})$ to the water waves equations \eqref{WaterWavesEquations} one has respectively
        \begin{align*}
            \begin{cases}
                \partial_t \zeta + \mathcal{N}_{(A)}^1(\zeta,\psi) = \mu^k \epsilon^l R_1, \\
                \partial_t \psi + \mathcal{N}_{(A)}^2(\zeta,\psi) = \mu^k \epsilon^l R_2,
                \end{cases}
                        \ \ \Or \ \ 
                \begin{cases}
                    \partial_t \zeta + \nabla\cdot(h\overline{V}) = 0, \\
                    \partial_t ((Id + \mu T_{(A')}[h]) \overline{V}) + \mathcal{N}_{(A')}^3(\zeta,\overline{V}) = \mu^k\epsilon^l R_3,
                \end{cases}
        \end{align*}
        with respectively $|R_1|_{H^s},|R_2|_{H^s}\leq N(s+n)$ or $|R_3|_{H^{s}}\leq N(s+n+1)$ on $[0,\frac{T}{\epsilon}]$. 
\end{mydef}
\noindent For sufficiently regular initial data satisfying hypotheses \ref{FundamentalHypothesis}, the existence and uniqueness of a solution of the water waves equations with existence time of order $1/\epsilon$ independent of $\mu$ and with the regularity we want is given by the theorem 4.16 in \cite{WWP}.

We now state our consistency results. 
\begin{mypp}\label{Proposition consistency first full dispersion GN system}
    Let $\mathrm{F}_1$ and $\mathrm{F}_2$ be the Fourier multipliers defined in proposition \ref{Proposition Approximation Vbar estimates}. The water waves equations \eqref{WaterWavesEquations} are consistent at order $O(\mu^2\epsilon)$ in the shallow water regime $\mathcal{A}$ (see definition \ref{Definition consistency}) with
    \begin{align}\label{FirstFullDispersionModel(InPsi)}
        \begin{cases} \partial_t \zeta +\nabla\cdot(h\nabla\psi) + \frac{\mu}{6}\big(\Delta(\mathrm{F}_2[h^3\Delta\psi]) + \Delta(h^3\mathrm{F}_2[\Delta\psi])\big)=0, \\
        \partial_t \psi  +\zeta + \frac{\epsilon}{2}|\nabla\psi|^2 -\frac{\mu\epsilon}{2}h^2(\mathrm{F}_2[\Delta\psi])\Delta\psi=0,
        \end{cases}
    \end{align}
with $n=4$.
\end{mypp}

\begin{mypp}\label{Proposition consistency second Full dispersion Gn System}
    Let $\mathrm{F}_3$ be the Fourier multiplier defined in proposition \ref{Proposition Approximation Vbar estimates}. Let $T[h]V \colonequals -\frac{1}{6h}(\nabla(h^3\mathrm{F}_3[\nabla \cdot V]) + \nabla(\mathrm{F}_3[h^3\nabla \cdot V]))$. The water waves equations are consistent at order $O(\mu^2\epsilon)$ in the shallow water regime $\mathcal{A}$ with
    \begin{align}\label{SecondFullDispersionModel(InVbar)}
    \begin{cases}
        \partial_t \zeta + \nabla \cdot(h\overline{V})=0, \\
        \partial_t((Id+\mu T[h])\overline{V}) +\nabla \zeta &+\frac{\epsilon}{2} \nabla(|\overline{V}|^2) - \frac{\mu\epsilon}{6}\nabla(\frac{1}{h}\overline{V}\cdot \nabla(h^3\mathrm{F}_3[\nabla\cdot\overline{V}]+\mathrm{F}_3[h^3\nabla\cdot\overline{V}])) \\
        &- \frac{\mu\epsilon}{2}\nabla(h^2\mathrm{F}_3[\nabla\cdot\overline{V}]\nabla\cdot\overline{V}) =0.
    \end{cases}
\end{align}
with $n=6$.
\end{mypp}
Setting $\mathrm{F_3} = I_d$ in \eqref{FirstFullDispersionModel(InPsi)} we recover the Green-Naghdi system introduced in \cite{CraigGroves94, Whitham67}. Setting $\mathrm{F_3} = I_d$ in \eqref{SecondFullDispersionModel(InVbar)} we recover the classical Green-Nagdhi system which has been proved to be consistent with precision $O(\mu^2)$ in \cite{WWP, LannesBonneton09}. We refer to \eqref{FirstFullDispersionModel(InPsi)} and \eqref{SecondFullDispersionModel(InVbar)} as full dispersion Green-Nagdhi systems.
\begin{myrem}
    The two full dispersion Green-Nagdhi systems \eqref{FirstFullDispersionModel(InPsi)} and \eqref{SecondFullDispersionModel(InVbar)} enjoy a canonical Hamiltonian formulation (see \eqref{First approximated hamiltonian} and \eqref{Second approximated hamiltonian}).
\end{myrem}

From these two previous propositions we are able to give results on the consistency with respect to water waves equations \eqref{WaterWavesEquations} of most of the full dispersion systems appearing in the literature. We give the examples of two systems that kept the author attention for there mathematical properties. The first one is a single layer, two dimensional generalisation with no surface tension of the model introduced in \cite{DucheneIsrawiTalhouk16} to study high-frequency Kevin-Helmholtz instabilities. That is
\begin{align}\label{FullDisperionGNSystmVincentIntro}
     \begin{cases}  
        \partial_t \zeta + \nabla \cdot(h\overline{V})=0, \\
        \partial_t(\overline{V}-\frac{\mu}{3h}\nabla(\sqrt{\mathrm{F}_3}h^3\sqrt{\mathrm{F}_3}[\nabla\cdot\overline{V}]) ) +\nabla \zeta &+\frac{\epsilon}{2} \nabla(|\overline{V}|^2) - \frac{\mu\epsilon}{3}\nabla(\frac{1}{h}\overline{V}\cdot \nabla(\sqrt{\mathrm{F}_3}h^3\sqrt{\mathrm{F}_3}[\nabla\cdot\overline{V}])) \\
        &- \frac{\mu\epsilon}{2}\nabla(h^2\mathrm{F}_3[\nabla\cdot\overline{V}]\nabla\cdot\overline{V}) =0.
    \end{cases}
\end{align}
\begin{mypp}\label{Consistency proposition Vincent Gn system}
    The water waves equations are consistent at order $O(\mu^2\epsilon)$ in the shallow water regime $\mathcal{A}$ with the system \eqref{FullDisperionGNSystmVincentIntro}.
\end{mypp}
For the same reason as system \eqref{SecondFullDispersionModel(InVbar)}, we refer to \eqref{FullDisperionGNSystmVincentIntro} as a full dispersion Green-Naghdi system.

The second one is a Whitham-Boussinesq system studied in \cite{DinvaySelbergTesfahun19}. They proved a local well-posedness result in dimension 2 and a global well-posedness result for small data in dimension 1. This system is
\begin{align}\label{Full Dispersion Boussinesq system Dinvay Intro}
    \begin{cases}
        \partial_t \zeta + \mathrm{F}_1\Delta\psi + \epsilon \mathrm{F}_1\nabla\cdot(\zeta \mathrm{F}_1\nabla\psi)=0,\\
        \partial_t\nabla\psi + \nabla\zeta + \frac{\epsilon}{2}\nabla(\mathrm{F}_1|\nabla\psi|)^2 = 0.
    \end{cases}
\end{align}
With the definition we gave of consistency (see definition \ref{Definition consistency}) we have the following proposition. 
\begin{mypp}\label{Consistency proposition Dinvay system}
The water waves equations \eqref{WaterWavesEquations} are consistent at order $O(\mu\epsilon)$ in the shallow water regime $\mathcal{A}$ with the system
\begin{align*}
    \begin{cases}
        \partial_t \zeta + \mathrm{F}_1\Delta\psi + \epsilon \mathrm{F}_1\nabla\cdot(\zeta \mathrm{F}_1\nabla\psi)=0,\\
        \partial_t\psi + \zeta + \frac{\epsilon}{2}(\mathrm{F}_1|\nabla\psi|)^2 = 0.
    \end{cases}
\end{align*}
\end{mypp}
\noindent But one can easily adapt definition \ref{Definition consistency} to say that the water waves equations \eqref{WaterWavesEquations} are consistent at order $O(\mu\epsilon)$ with system \eqref{Full Dispersion Boussinesq system Dinvay Intro}.

\begin{myrem}
    At the best of the author's knowledge, in the case of the full dispersion Green-Naghdi systems, before this work it was only known that the water waves equations are consistent in the shallow water regime with respect to system \eqref{FullDisperionGNSystmVincentIntro} at order $O(\mu^2)$ at worse (proposition 5.7 in \cite{DucheneIsrawiTalhouk16}), that is the same precision as the one of the classical Green-Naghdi models, see chapter 5 in \cite{WWP}. The use of proposition \ref{Proposition Approximation Vbar estimates} allow us to improve the precision order by a factor $\epsilon$, as conjectured in \cite{DucheneIsrawiTalhouk16}. So that in a situation in which $\epsilon \sim \mu$ (long wave regime) we gain a power of $\mu$, i.e. the full dispersion Green-Naghdi systems are precised at order $O(\mu^3)$ in the long wave regime, in the sense of consistency. Moreover even if $\mu$ is not so small, the latter systems stay good approximations of the water waves equations as long as $\epsilon$ is small enough, making them more robust than the corresponding classical Green-Naghdi models.
    
    The case of the Whitham-Boussinesq systems is a bit more subtle. Indeed, using the same argument as in the proof of proposition 5.7 in \cite{DucheneIsrawiTalhouk16}, one would obtain, in the shallow water regime, a precision order of $O(\mu^2 + \mu\epsilon)$ for these systems, that is the same as the one of the Boussinesq models, see chapter 5 in \cite{WWP}. In this paper we prove that the precision order of the Whitham-Boussinesq systems is in fact $O(\mu\epsilon)$. So that the improvement can only be seen in a regime in which $\epsilon \ll \mu$. It still makes them more robust than the Boussinesq models for the same reason as the full dispersion Green-Naghdi systems.
\end{myrem}

\subsection{outline}

In section \ref{Shallow water approximation of the vertically averaged horizontal component of the velocity} we prove proposition \ref{Proposition Approximation Vbar estimates}. We begin in subsection \ref{Formal construction} by constructing the shallow water expansions appearing in the proposition using a formal reasoning, see lemma \ref{Formal expansions of phi}. Then we use, in subsection \ref{Rigorous expansions}, the fact that we have explicit candidates as approximations of the velocity potential to prove the estimates of proposition \ref{Proposition Approximation Vbar estimates}.

In section \ref{Derivation and consistency of a first full dispersion Green-Naghdi system} we focus on system \eqref{FirstFullDispersionModel(InPsi)}. First, in subsection \ref{Formal Derivation}, we derive formally \eqref{FirstFullDispersionModel(InPsi)} from Hamilton's equations associated with an approximated Hamiltonian based on proposition \ref{Proposition Approximation Vbar estimates}. In subsection \ref{Consistency with respect to the water waves system} we prove rigorously the consistency of the water waves equations with system \eqref{FirstFullDispersionModel(InPsi)}, i.e. we prove proposition \ref{Proposition consistency first full dispersion GN system}.

In section \ref{Derivation and consistency of a second full dispersion Green-Naghdi system} we focus on system \eqref{SecondFullDispersionModel(InVbar)} and do the same process as for system \eqref{FirstFullDispersionModel(InPsi)}. In subsection \ref{Formal Derivation 2} we get a second approximated Hamiltonian of the water waves system \eqref{WaterWavesEquations} and derive the Hamilton equations associated with, giving \eqref{SecondFullDispersionModel(InVbar)}. In subsection \ref{Consistency with respect to the water waves equations} we prove proposition \ref{Proposition consistency second Full dispersion Gn System}.

In section \ref{Consistency of the full dispersion models appearing in the literature} we prove the consistency of the water waves equations with the systems \eqref{FullDisperionGNSystmVincentIntro} and \eqref{Full Dispersion Boussinesq system Dinvay Intro}. In subsection \ref{Full dispersion Green-Naghdi system} we make use of proposition \ref{Proposition consistency first full dispersion GN system} to prove proposition \ref{Consistency proposition Vincent Gn system}. And in subsection \ref{Full dispersion Boussinesq systems} we use \ref{Proposition consistency second Full dispersion Gn System} to prove proposition \ref{Consistency proposition Dinvay system}.

\section{Shallow water approximation of the vertically averaged horizontal component of the velocity}
\label{Shallow water approximation of the vertically averaged horizontal component of the velocity}

\subsection{Formal construction}
Here, we construct formally two different approximations of the velocity potential $\phi$ (see definition \ref{Definitions Fondamentales}) at order $O(\mu^2\epsilon)$ (see just below notation \ref{order O mu squared epsilon}) then deducing ones on the vertically averaged horizontal velocity $\overline{V}$ (see definition \ref{Definitions Fondamentales}) with the same order of precision in terms of the trace at the surface of the velocity potential $\psi$ (see also definition \ref{Definitions Fondamentales}). And we also construct an approximation of this last quantity in terms of $\overline{V}$. For these purposes, we use a method similar to the one developed in chapter 5 of \cite{WWP}. Everything can be proved in a functional framework and rigorous results will be provided in subsection 2.2.

Before starting the reasoning, for the sake of clarity we introduce a notation.
\begin{mynot}\label{order O mu squared epsilon}
    Let $k\in\mathbb{N}$ and $ l\in\mathbb{N}$. In all this paper, a function $R$ is said to be of order $O(\mu^k\epsilon^l)$ if divided by $\mu^k\epsilon^l$ this function is uniformly bounded with respect to $(\mu,\epsilon) \in \mathcal{A}$ (defined in definition \ref{Definitions Fondamentales}) in some Sobolev norm.
\end{mynot}

Let us also recall (see again definition \ref{Definitions Fondamentales}) that by definition the velocity potential $\phi$ satisfies an elliptic problem in the flat strip $\mathcal{S} = \mathbb{R}^d \times (-1,0)$.\\
This problem is written in term of the velocity potential at the surface $\psi$:
\begin{align}\label{Laplace Problem in the flat strip}
    \begin{cases} \nabla^{\mu} \cdot P(\Sigma_t)\nabla^{\mu} \phi = 0, \\
    \phi_{|_{z=0}} = \psi \ \ , \ \ \partial_z \phi_{|_{z=-1}}=0,
    \end{cases}
\end{align}
where $\nabla^{\mu}$ stands for the $(d+1)$-gradient operator defined by $\nabla^{\mu} = (\sqrt{\mu}\nabla^T,\partial_z)^T$,\\ and $P(\Sigma_t) =  \begin{pmatrix} (1+\epsilon \zeta)I_d & -\sqrt{\mu}\epsilon(z+1)\nabla\zeta \\ -\sqrt{\mu}\epsilon(z+1)\nabla\zeta^{T} & \frac{1 + \mu \epsilon^2 (z+1)^2|\nabla\zeta|^2}{1+\epsilon\zeta} \end{pmatrix}$.

\indent 

Now we begin the constructions.\\
\underline{Step 1}: The first step is to find approximations of $\phi$, which satisfy the elliptic problem \eqref{Laplace Problem in the flat strip} up to terms of order $O(\mu^2 \epsilon)$. The functional meaning will be precised in the next subsection but can already be anticipated, we will work with Sobolev and Beppo-Levi spaces (see notations \ref{functional framework}).
\begin{mylem}\label{Formal expansions of phi}
Let $\mathrm{F}_0 = \frac{\cosh((z+1)\sqrt{\mu}|D|)}{\cosh(\sqrt{\mu}|D|)}$ be a Fourier multiplier depending on the transversal variable $z\in [-1,0]$. Let $\phi$ be the solution of the Laplace problem in the flat strip \eqref{Laplace Problem in the flat strip}. We have the formal expansions 
\begin{align*}
    \begin{cases}
        \phi_0 \colonequals \mathrm{F}_0 \psi = \phi + O(\mu\epsilon),\\ 
        \phi_{\appr} \colonequals \mathrm{F}_0 \psi - \mu \epsilon \zeta (1+h)(\frac{z^2}{2}+z)\Delta\psi = \phi + O(\mu^2\epsilon), \\
        \widetilde{\phi}_{\appr} \colonequals \psi + h^2(\mathrm{F}_0-1)\psi = \phi + O(\mu^2\epsilon).
    \end{cases}
\end{align*}
\end{mylem}
\begin{proof} 
The idea is to do a multi-scale expansion for the solution of the elliptic problem \eqref{Laplace Problem in the flat strip} by approximately solving \eqref{Laplace Problem in the flat strip} and using the technical lemma \ref{lemma 3.43}.

Let us remark that multiplying the elliptic equation of \eqref{Laplace Problem in the flat strip} by the depth $h = 1+\epsilon \zeta$ allow us to decompose it into two parts:
\begin{align}\label{Decomposition straightened Laplacian}
    h\nabla^{\mu} \cdot P(\Sigma_t)\nabla^{\mu} \phi = (\partial_{z}^{2}\phi + \mu \Delta\phi) + \mu \epsilon A(\nabla,\partial_z)[\phi],
\end{align}
where $A(\nabla,\partial_z)$ is an operator defined as follow
\begin{multline}\label{definition operator A}
    A(\nabla,\partial_z)[\phi] = \nabla \cdot (\zeta \nabla\phi) + \zeta \nabla \cdot((1+\epsilon \zeta) \nabla\phi)
    + \epsilon |\nabla \zeta|^2 \partial_{z}((z+1)^2\partial_z \phi) \\ - (1+\epsilon \zeta)(z+1)\nabla \cdot (\nabla \zeta \partial_z \phi)
    - (1+\epsilon \zeta)\nabla\zeta \cdot \partial_z( (z+1) \nabla\phi).
\end{multline}

In the elliptic problem \eqref{Laplace Problem in the flat strip} let's only consider the part which is not of order $O(\mu\epsilon)$ and denote $\phi_0$ its solution, i.e. $\phi_0$ is the solution of the problem
\begin{align}\label{Laplace problem for phi0}
    \begin{cases}
    \partial_{z}^2 \phi_0 + \mu\Delta\phi_0 =0, \\
    \phi_{0|_{z=0}} = \psi \ \ , \ \ \partial_z \phi_{0|_{z=-1}}=0.
    \end{cases}
\end{align}
We get the expression of $\phi_0$ by a Fourier analysis:
\begin{align}\label{phi0}
    \phi_0 = \frac{\cosh((z+1)\sqrt{\mu}|D|)}{\cosh(\sqrt{\mu}|D|)} \psi.
\end{align}
Thus $\phi_0$ is defined as a bounded Fourier multiplier applied to the trace at the surface of the velocity potential. And by lemma \ref{lemma 3.43}, $\phi = \phi_0 + O(\mu\epsilon)$. \\
Now we seek $\phi_1$ so that $\phi = \phi_0 + \mu \epsilon \phi_1 + O(\mu^2 \epsilon)$. \eqref{Decomposition straightened Laplacian} and the lemma \ref{lemma 3.43} tell us that we just have to ask $\phi_1$ to be the solution of the following problem:
\begin{align}\label{Laplace problem for phi1}
    \begin{cases}
    \partial_{z}^2\phi_1 = -A(\nabla,\partial_z)[\phi_0], \\
    \phi_{1|_{z=0}} = 0 \ \ , \ \ \partial_z \phi_{1|_{z=-1}}=0.
    \end{cases}
\end{align}
To solve \eqref{Laplace problem for phi1}, we integrate two times with respect to the transversal variable $z$, and we simplify the result thanks to the next formal property.

Let $\mathrm{F}_0$ be the Fourier multiplier $\frac{\cosh((z+1)\sqrt{\mu}|D|)}{\cosh(\sqrt{\mu}|D|)}$. Then for any $z\in [-1,0]$ we have
\begin{align}\label{Fourier Multipliers Expansions}
    \begin{cases}
        \frac{1-\mathrm{F}_0}{\mu|D|^2} = -\frac{z^2}{2} -z +O(\mu) \ \ , \ \ 1-(z+1)^2 \mathrm{F}_0 = -z^2-2z + O(\mu), \\
        \frac{\tanh(\sqrt{\mu}|D|)}{\sqrt{\mu}|D|} = 1 + O(\mu) \ \ , \ \ \frac{z+1}{\sqrt{\mu}|D|} \frac{\sinh((z+1)\sqrt{\mu}|D|)}{\cosh(\sqrt{\mu}|D|)} = (z+1)^2 + O(\mu).
    \end{cases}
\end{align}
See proposition \ref{FourierMultiplierApproximations} for a rigorous proof of these expansions. \\
This allows us after computations to obtain a simple approximation of $\phi_1$:
\begin{align*} 
    \phi_1 = -\zeta(1+h) (\frac{z^2}{2}+z)\Delta\psi + O(\mu).
\end{align*}
Thus we obtain the following expression of a first approximation $\phi_{app}$ of the velocity potential $\phi$ at order $O(\mu^2\epsilon)$:
\begin{align}\label{first expression of phi_app}
    \phi_{app} = \mathrm{F}_0 \psi - \mu \epsilon \zeta (1+h)(\frac{z^2}{2}+z)\Delta\psi = \phi + O(\mu^2\epsilon).
\end{align}
Moreover using the first approximation of \eqref{Fourier Multipliers Expansions} we get a second approximation of $\phi$:
\begin{align}\label{Second approximation of phi}
    \widetilde{\phi}_{app}= \psi + h^2(\mathrm{F}_0-1)\psi = \phi + O(\mu^2\epsilon).
\end{align}
\end{proof}

\noindent \underline{Step 2}: The second step is to use both approximations $\phi_{\appr}$ and $\widetilde{\phi}_{\appr}$ to get approximations of the vertically averaged horizontal velocity $\overline{V}$ at order $O(\mu^2\epsilon)$.
\begin{mypp}\label{overlineV formal approximations}
    Let $\mathrm{F_1} \colonequals \frac{\tanh{(\sqrt{\mu}|D|)}}{\sqrt{\mu}|D|}$ and $\mathrm{F_2} = \frac{3}{\mu|D|^2}(1- \frac{\tanh{(\sqrt{\mu}|D|)}}{\sqrt{\mu}|D|})$ be two Fourier multipliers. Let $\overline{V}$ be the vertically averaged horizontal velocity. We have the formal expansions:
    \begin{align*}
        \begin{cases}
                \overline{V}_{\appr} \colonequals \mathrm{F}_1\nabla\psi + \frac{\mu \epsilon}{3}[h \nabla\zeta \Delta\psi + \nabla(\zeta(1+h)\Delta\psi)] = \overline{V} + O(\mu^2\epsilon), \\
                \widetilde{\overline{V}}_{app} \colonequals \nabla\psi + \frac{\mu}{3h}\nabla(h^3 \mathrm{F}_2 \Delta\psi) = \overline{V} + O(\mu^2\epsilon).
        \end{cases} 
    \end{align*}
\end{mypp}
\begin{proof}
For that purpose, we use the expression of $\overline{V}$ in term of the velocity potential $\psi$ (see definition \ref{Definitions Fondamentales}), that is:
\begin{align}\label{Expression of Vbar}
    \overline{V} = \frac{1}{h} \int_{-1}^{0} [h\nabla\phi - \epsilon(z+1)\nabla\zeta\partial_z\phi] \dz.
\end{align}
Replacing $\phi$ by $\phi_{app}$ in it, we get an approximation $\overline{V}_{\appr,0}$ of $\overline{V}$:
\begin{align*}
   \overline{V}_{\appr,0} = & \frac{\tanh(\sqrt{\mu}|D|)}{\sqrt{\mu}|D|} \nabla\psi - \epsilon \frac{\nabla\zeta}{h}(\psi-\frac{\tanh(\sqrt{\mu}|D|)}{\sqrt{\mu}|D|}\psi) \\
    + & \frac{\mu \epsilon}{3}[\frac{\nabla\zeta}{h} \epsilon \zeta(1+h) \Delta\psi + \nabla(\zeta(1+h)\Delta\psi)] = \overline{V} + O(\mu^2\epsilon). 
\end{align*}
Moreover remark that formally we also have the following expansion 
\begin{align*}
    \psi-\frac{\tanh(\sqrt{\mu}|D|)}{\sqrt{\mu}|D|}\psi = -\frac{\mu}{3}\Delta\psi + O(\mu^2).
\end{align*}
Using it we get a new approximation of $\overline{V}$, denoted $\overline{V}_{\appr}$:
\begin{align}\label{first approximation of bar V}
    \overline{V}_{\appr} = \mathrm{F}_1\nabla\psi + \frac{\mu \epsilon}{3}[h \nabla\zeta \Delta\psi + \nabla(\zeta(1+h)\Delta\psi)] = \overline{V} + O(\mu^2\epsilon),
\end{align}
where $\mathrm{F}_1 = \frac{\tanh(\sqrt{\mu}|D|)}{\sqrt{\mu}|D|}$.

Replacing $\phi$ by $\widetilde{\phi}_{app}$ in \eqref{Expression of Vbar} we get another useful approximation of $\overline{V}$ denoted $\widetilde{\overline{V}}_{\appr}$:
\begin{align}\label{VbarTildeApp expression}
   \widetilde{\overline{V}}_{\appr}
    = \nabla\psi + \frac{1}{h} \nabla(h^3(\frac{\tanh(\sqrt{\mu}|D|)}{\sqrt{\mu}|D|}-1)\psi).
\end{align}
Let $\mathrm{F}_2$ be the Fourier multiplier such that
\begin{align}\label{DefinitionF_2}
    \frac{\tanh(\sqrt{\mu}|D|)}{\sqrt{\mu}|D|}-1 = -\frac{\mu}{3}|D|^2 \mathrm{F}_2,
\end{align}
then we write
\begin{align}\label{second approximation of V bar}
    \widetilde{\overline{V}}_{app} = \nabla\psi + \frac{\mu}{3h}\nabla(h^3 \mathrm{F}_2 \Delta\psi) = \overline{V} + O(\mu^2\epsilon). 
\end{align}
\end{proof}

\noindent \underline{Step 3}: We construct approximations of $\nabla\psi$ in terms of $\overline{V}$ at order $O(\mu^2\epsilon)$.
\begin{mypp}
    Let $\mathrm{F_3} = \mathrm{F_2}\circ\mathrm{F_1}^{-1}$ be a Fourier multiplier where $\mathrm{F}_1$ and $\mathrm{F}_2$ are defined in proposition \ref{overlineV formal approximations}. Then we have the formal expansion:
    \begin{align*}
            \nabla\psi= \overline{V} - \frac{\mu}{3h}\nabla(h^3 \mathrm{F}_3 \nabla \cdot \overline{V}) + O(\mu^2\epsilon).
    \end{align*}
\end{mypp}
\begin{proof}
Using \eqref{first approximation of bar V} in \eqref{second approximation of V bar} we obtain
\begin{align*}
    \nabla\psi= \overline{V} - \frac{\mu}{3h}\nabla(h^3 \mathrm{F}_2 \mathrm{F}_1^{-1} \nabla \cdot \overline{V}) + O(\mu^2\epsilon).
\end{align*}
Let $\mathrm{F}_3$ be the Fourier multiplier defined by $\mathrm{F}_3 \colonequals \mathrm{F}_2 \mathrm{F}_1^{-1}$. Then
\begin{align}\label{approximation of phi in terms of V bar}
    \nabla\psi= \overline{V} - \frac{\mu}{3h}\nabla(h^3 \mathrm{F}_3 \nabla \cdot \overline{V}) + O(\mu^2\epsilon).
\end{align}
\end{proof}
\begin{myrem}
From \eqref{first approximation of bar V}, \eqref{second approximation of V bar} and \eqref{approximation of phi in terms of V bar} we obtain formally the approximations displayed in proposition \ref{Proposition Approximation Vbar estimates}. \\
Again the rigorous proof is given in the following subsection.
\end{myrem}

\begin{myrem}
    The simplifications given by \eqref{Fourier Multipliers Expansions} allowed us to write simple approximations. Omitting this step of simplification in \underline{Step 1}, we introduced the beginning of an iterative process which allows to construct approximations of order $O(\mu^k\epsilon)$ for any $k \in \mathbb{N}^*$.
\end{myrem}

\label{Formal construction}

\subsection{Rigorous expansions}
\noindent In this subsection we prove rigorously the estimations of proposition \ref{Proposition Approximation Vbar estimates}.

\noindent For convenience we rewrite proposition \ref{Proposition Approximation Vbar estimates} using \eqref{first approximation of bar V}, \eqref{second approximation of V bar} and \eqref{approximation of phi in terms of V bar}.

\begin{mypp}\label{Proposition Approximation Vbar estimates with VbarApprox}
    Let $s \geq 0$, and $\zeta \in H^{s+4}(\mathbb{R}^d)$ be such that \ref{FundamentalHypothesis} is satisfied. Let $\psi \in \dot{H}^{s+5}(\mathbb{R}^d)$, and $\overline{V}$ be as in \eqref{VerticallyAveragedHorizontalVelocity}. Let $\overline{V}_{\appr}$ be as in \eqref{first approximation of bar V} and $\widetilde{\overline{V}}_{\appr}$ be as in \eqref{second approximation of V bar}. Let also $\mathrm{F}_1 \colonequals \frac{\tanh{(\sqrt{\mu}|D|)}}{\sqrt{\mu}|D|}$, $\mathrm{F}_2 \colonequals \frac{3}{\mu|D|}(1- \frac{\tanh{(\sqrt{\mu}|D|)}}{\sqrt{\mu}|D|})$, and $\mathrm{F}_3 = \mathrm{F}_2\mathrm{F}_1^{-1}$ be three Fourier multipliers.\\
    The following estimates hold:
    \begin{align}\label{AsyptoticDevelopmentVbar}
        \begin{cases}
            |\overline{V}-\mathrm{F}_1\nabla\psi|_{H^{s}} \leq \mu\epsilon M(s+3)|\nabla\psi|_{H^{s+2}},\\
            |\overline{V}-\overline{V}_{\appr}|_{H^s} \leq \mu^2\epsilon M(s+3)|\nabla\psi|_{H^{s+4}}, \\
            |\overline{V}-\widetilde{\overline{V}}_{\appr}|_{H^{s}} \leq \mu^2\epsilon M(s+3) |\nabla\psi|_{H^{s+4}},\\
            |\overline{V}-\nabla\psi -\frac{\mu}{3h}\nabla(h^3 \mathrm{F}_3 \nabla \cdot \overline{V})|_{H^s} \leq \mu^2\epsilon M(s+4) |\nabla\psi|_{H^{s+4}}.
        \end{cases}
    \end{align}
\end{mypp}

\indent 

\noindent Every steps in the previous subsection can be justified rigorously. But we won't follow the same path. We will mainly use the fact that we have explicit candidates for each approximations. It will give us sharper estimations.

\noindent \underline{Step 1}: The first step is to write and prove a rigorous version of lemma \ref{Formal expansions of phi}.
\begin{mypp}
    Let $t_0 > d/2$, and $s \geq 0$. Let also $\phi_{\appr}$ and $\widetilde{\phi}_{\appr}$ be defined in lemma \ref{Formal expansions of phi}. We have the following estimates:
    \begin{align*}
        \begin{cases}
            ||\Lambda^s\nabla^{\mu}(\phi-\widetilde{\phi}_{\appr})||_2 \leq \mu^2\epsilon M(s+2)|\nabla\psi|_{H^{s+3}}, \\
            ||\Lambda^s\nabla^{\mu}(\phi-\phi_{\appr})||_2 \leq \mu^2\epsilon M(s+2)|\nabla\psi|_{H^{s+3}}.
        \end{cases}
    \end{align*}
\end{mypp}
\begin{proof}
We begin by computing the straightened Laplacian of $\widetilde{\phi}_{\appr}$, $h\nabla^{\mu}\cdot P(\Sigma_t)\nabla^{\mu}\widetilde{\phi}_{\appr}$ (see definition \ref{Definitions Fondamentales} for the expression of $P(\Sigma_t)$). Let $\mathrm{F}_0$ be the Fourier multiplier $\frac{\cosh{((z+1)\sqrt{\mu}|D|)}}{\cosh{(\sqrt{\mu}|D|)}}$ for any $z \in [-1,0]$. Recalling $\widetilde{\phi}_{\appr} = \psi + h^2(\mathrm{F}_0-1)\psi = \mathrm{F}_0\psi + (h^2-1)(\mathrm{F}_0-1)\psi$ and  $\partial_z^2 \phi_0 + \mu\Delta\phi_0 = 0$ (see \eqref{Laplace problem for phi0}), we get
\begin{align*}
    h\nabla^{\mu}\cdot P(\Sigma_t)\nabla^{\mu}\widetilde{\phi}_{\appr} = &\mu(h^2-1)(\mathrm{F}_0-1)\Delta\psi + \mu\Delta((h^2-1)(\mathrm{F}_0-1)\psi) \\
    + &\mu\epsilon A(\nabla,\partial_z) (\mathrm{F}_0-1)\psi + \mu\epsilon A(\nabla,\partial_z)(h^2-1)(\mathrm{F}_0-1)\psi
\end{align*}
(see \eqref{definition operator A} for the definition of operator A).
We estimate it thanks to product estimates \ref{product estimate} and the following estimations on $\mathrm{F}_0$ (where $a \lesssim b$ means there exists a constant $C > 0$ independent of $\mu$ such that $a \leq C b$)
\begin{align*}
        ||\Lambda^s(\mathrm{F}_0 - 1)\psi||_2 \lesssim \mu|\nabla\psi|_{H^{s+1}}, \ \
        ||\Lambda^s\partial_z \mathrm{F}_0\psi||_2 \lesssim \mu|\nabla\psi|_{H^{s+1}}, \ \
        ||\Lambda^s\partial_z^2 \mathrm{F}_0\psi||_2 \lesssim \mu|\nabla\psi|_{H^{s+1}},
\end{align*}
stemming from the existence of $C>0$ such that for any $z\in [-1,0]$, $\xi \in\mathbb{R}^d$
\begin{align*}
    |F_0(z,\xi)-1| + |\partial_z F_0(z,\xi)| + |\partial_z^2F_0(z,\xi)|\leq C|\xi|^2.
\end{align*}
We get
\begin{align}\label{Estimation Straightened Laplacian phitildeappr}
    ||\Lambda^s h\nabla^{\mu}\cdot P(\Sigma_t)\nabla^{\mu}\widetilde{\phi}_{\appr}||_2 \leq \mu^2\epsilon M(s+2)|\nabla\psi|_{H^{s+3}}.
\end{align}
Now let's define the function $\widetilde{u} \colonequals \phi - \widetilde{\phi}_{\appr}$. It solves the following elliptic problem
\begin{align*}
    \begin{cases}
        h\nabla^{\mu}\cdot P(\Sigma)\nabla^{\mu} \tilde{u}=-\mu^2\epsilon R, \\
        \tilde{u}|_{z=0}=0, \ \ \partial_z \tilde{u}|_{z=-1}=0,
    \end{cases}
\end{align*}
where $R = \frac{1}{\mu^2\epsilon} h\nabla^{\mu}\cdot P(\Sigma_t)\nabla^{\mu}\widetilde{\phi}_{\appr}$.
So \eqref{Estimation Straightened Laplacian phitildeappr} gives a control on the remainder and we get from lemma \ref{lemma 3.43} one of the wanted estimations: 
\begin{align}\label{Estimation on utilde}
    ||\Lambda^s\nabla^{\mu}(\phi-\widetilde{\phi}_{\appr})||_2 = ||\Lambda^s \nabla^{\mu} \widetilde{u}||_2 \leq \mu^2\epsilon M(s+1) ||\Lambda^s R||_2 \leq \mu^2\epsilon M(s+2)|\nabla\psi|_{H^{s+3}}.
\end{align}

\noindent Proceeding similarly as for \eqref{Estimation on utilde} we get the estimates on $\phi_{app}$ defined by \eqref{first expression of phi_app}
\begin{align}\label{Estimation on u}
     ||\Lambda^s\nabla^{\mu}(\phi-\phi_{\appr})||_2 \leq \mu^2\epsilon M(s+2)|\nabla\psi|_{H^{s+3}}.
\end{align}
\end{proof}
\noindent \underline{Step 2}: From the estimates on $\widetilde{\phi}_{\appr}$ and $\phi_{\appr}$ we get the ones on the error made when approximating $\overline{V}$ by $\widetilde{\overline{V}}_{\appr}$ or $\overline{V}_{\appr}$ using 
\begin{align}\label{Definition of Vbar and Vbartildeapp}
    \begin{cases}
        \overline{V} = \frac{1}{h} \int_{-1}^{0} [h\nabla\phi - (z\nabla h + \epsilon\nabla\zeta)\partial_z\phi] \dz, \\
        \widetilde{\overline{V}}_{\appr} = \frac{1}{h} \int_{-1}^{0} [h\nabla\widetilde{\phi}_{\appr} - (z\nabla h + \epsilon\nabla\zeta)\partial_z\widetilde{\phi}_{\appr}] \dz,\\
        \overline{V}_{\appr} = \frac{1}{h} \int_{-1}^{0} [h\nabla\phi_{\appr} - (z\nabla h + \epsilon\nabla\zeta)\partial_z\phi_{\appr}] \dz.
    \end{cases}
\end{align}
Indeed for any $u$ sufficiently regular and vanishing at $z=0$ we have, thanks to Jensen inequality and Poincaré inequality (page 40 of \cite{WWP}):
\begin{align*}
    |\int_{-1}^{0} u \dz|_{H^s}^2 = \int_{\mathbb{R}^d} |\Lambda^s \int_{-1}^0 u \dz|^2 \dX \leq \int_{\mathbb{R}^d} (\int_{-1}^{0} |\Lambda^s u| \dz)^2 \dX \leq &\int_{-1}^{0} \int_{\mathbb{R}^d}|\Lambda^s u|^2 \dX \dz \\
    \leq &||\Lambda^s \partial_z u||^2_2 \leq ||\Lambda^s \nabla^{\mu}u||^2_2.
\end{align*}
Applying this last inequality to $\overline{V}-\widetilde{\overline{V}}_{\appr}$ and $\overline{V}-\overline{V}_{\appr}$ gives the desired estimations of proposition \ref{Proposition Approximation Vbar estimates with VbarApprox}
\begin{align}\label{Error estimation Vbartilde}
    \begin{cases}
    |\overline{V}-\widetilde{\overline{V}}_{\appr}|_{H^s} \leq M(s+1) ||\Lambda^{s+1}\nabla^{\mu}(\phi-\widetilde{\phi}_{\appr})||_2 \leq \mu^2\epsilon M(s+3)|\nabla\psi|_{H^{s+4}},\\
     |\overline{V}-\overline{V}_{\appr}|_{H^s} \leq M(s+1) ||\Lambda^{s+1}\nabla^{\mu}(\phi-\phi_{\appr})||_2 \leq \mu^2\epsilon M(s+3)|\nabla\psi|_{H^{s+4}}.
    \end{cases}
\end{align}

\indent 

\noindent \underline{Step 3}: We now prove the error estimates of proposition \ref{Proposition Approximation Vbar estimates with VbarApprox} on the approximation of $\nabla\psi$ by $\overline{V}$, i.e.
\begin{align*}
    |\overline{V}-\nabla\psi -\frac{\mu}{3h}\nabla(h^3 \mathrm{F}_3 \nabla \cdot \overline{V})|_{H^s} \leq \mu^2\epsilon M(s+4) |\nabla\psi|_{H^{s+4}}.
\end{align*}
The first thing we need is an estimation on $\phi-\phi_0$ (see \eqref{phi0}).\\
The straightened laplacian of $\phi_0$ is $\mu\epsilon A(\nabla,\partial_z)\mathrm{F}_0\psi$. So by the same reasoning as above (see \eqref{Estimation on utilde}) we get
\begin{align*}
    ||\Lambda^s\nabla^{\mu}(\phi-\phi_0)||_2 \leq M(s+1) ||\Lambda^s \mu\epsilon A(\nabla,\partial_z)\mathrm{F}_0 \psi||_2 \leq \mu\epsilon M(s+2)|\nabla\psi|_{H^{s+1}},
\end{align*}
(Because $\forall z \in [-1,0]$, $\forall \xi \in \mathbb{R}^d$, $\mathrm{F}_0(z,\sqrt{\mu}|\xi|) \leq 1$). \\
And if we define $\overline{V}_0 \colonequals \frac{1}{h} \int_{-1}^{0} [h\nabla\phi_0 - \epsilon(z+1)\nabla\zeta\partial_z\phi_0] \dz$ then using again Poincaré inequality we have
\begin{align}\label{Estimation Vbar moins Vbar0}
    |\overline{V}-\overline{V}_0|_{H^s} \leq M(s+1) ||\Lambda^{s+1}\nabla^{\mu}(\phi-\phi_0)||_2 \leq \mu\epsilon M(s+3)|\nabla\psi|_{H^{s+2}}.
\end{align}
Then remarking the following equality
\begin{align*}
    \mathrm{F}_1\nabla\psi = \frac{1}{h}\int_{-1}^{0}h\nabla\phi_0 \dz = \overline{V}_{0} + \frac{1}{h} \int_{-1}^0 \epsilon(z+1)\nabla\zeta\partial_z \phi_0 \dz,
\end{align*}
from the previous estimation on $\overline{V}-\overline{V}_0$ \eqref{Estimation Vbar moins Vbar0}, direct computations, product estimates \ref{product estimate} and quotient estimates \ref{Quotient estimate} we get
\begin{align*}
    |\overline{V}-\mathrm{F}_1\nabla\psi|_{H^{s}} \leq &|\overline{V} -\overline{V}_0|_{H^{s}} + |\frac{1}{h} \int_{-1}^0 \epsilon(z+1)\nabla\zeta\partial_z \phi_0 \dz|_{H^{s}} \\
    \leq &\mu\epsilon M(s+3) |\nabla\psi|_{H^{s+2}} + |\frac{\epsilon\nabla\zeta}{h}\int_{-1}^0(z+1)\partial_z \mathrm{F}_0\psi|_{H^{s}} \\
    \leq &\mu\epsilon M(s+3) |\nabla\psi|_{H^{s+2}}.
\end{align*}
This proves one of the inequality of proposition \ref{Proposition Approximation Vbar estimates with VbarApprox}. \\
Now using the error estimates on $\widetilde{\overline{V}}_{\appr}$ \eqref{Error estimation Vbartilde} and the upper bound $F_2(\sqrt{\mu}|\xi|) \leq \frac{1}{1+\frac{\mu|\xi|^2}{3}}$ we get
\begin{align*}
    &|\overline{V}-\nabla\psi -\frac{\mu}{3h}\nabla(h^3 \mathrm{F}_2 \nabla \cdot \mathrm{F}_1^{-1} \overline{V})|_{H^s} \\
    \leq &|\overline{V} - \nabla\psi -\frac{\mu}{3h}\nabla(h^3 \mathrm{F}_2 \nabla\cdot\nabla\psi)|_{H^s} + |\frac{\mu}{3h} \nabla(h^3 \mathrm{F}_2\nabla\cdot(\mathrm{F}_1^{-1}\overline{V}-\nabla\psi))|_{H^s} \\
    \leq &\mu^2\epsilon M(s+3)|\nabla\psi|_{H^{s+4}} + \mu M_0 |\mathrm{F}_1^{-1}\overline{V}-\nabla\psi|_{H^{s}}.
\end{align*}
By using the upper bound $F_1^{-1}(\sqrt{\mu}|\xi|) \leq 1 + \sqrt{\mu}|\xi|$ we get
\begin{align*}
    |\mathrm{F}_1^{-1}\overline{V}-\nabla\psi|_{H^{s}} \leq |\overline{V}-\mathrm{F}_1\nabla\psi|_{H^{s+1}}.
\end{align*}
And at the end we proved 
\begin{align}\label{approximation of nabla psi by Vbar}
    |\overline{V}-\nabla\psi -\frac{\mu}{3h}\nabla(h^3 F_2 \nabla \cdot F_1^{-1} \overline{V})|_{H^s} \leq \mu^2\epsilon M(s+4) |\nabla\psi|_{H^{s+4}}
\end{align}
This conclude the proof of proposition \ref{Proposition Approximation Vbar estimates with VbarApprox}.\\
\label{Rigorous expansions}

\section{Derivation and consistency of a first full dispersion Green-Naghdi system}
\label{Derivation and consistency of a first full dispersion Green-Naghdi system}

\subsection{Formal Derivation}
Let $H$ be the Hamiltonian of the Zakharov/Craig-Sulem's formulation of the water waves problem \eqref{WaterWavesEquations}:
\begin{align}
    H = \frac{1}{2}\int_{\mathbb{R}^d} \zeta^2 dX + \frac{1}{2\mu} \int_{\mathbb{R}^d} \psi \mathcal{G}^{\mu}\psi \dX,
\end{align}
where $\zeta$ is the surface elevation, $\psi$ is the velocity potential at the surface and $\mathcal{G}^{\mu}$ is the Dirichlet-to-Neumann operator. Let us recall that the expression of the Hamilton equations derived from an Hamiltonian, here $H(\zeta,\psi)$, is
\begin{align*}
    \begin{cases}
        \partial_t \zeta = \delta_{\psi} H,\\
        \partial_t \psi = -\delta_{\zeta} H,
    \end{cases}
\end{align*}
where $\delta_{\zeta}$ and $\delta_{\psi}$ are functional derivatives.\\
Using the definition of $\mathcal{G}^{\mu}$ in term of $\overline{V}$ \eqref{Dirichlet-to-NeumannOperator} and a formal integration by parts, we get
\begin{align}
    H = \frac{1}{2} \int_{\mathbb{R}^d} \zeta^2 \dX + \frac{1}{2} \int_{\mathbb{R}^d} h\overline{V} \cdot \nabla\psi \dX.
\end{align}
Replacing $\overline{V}$ by one of its approximation $\widetilde{\overline{V}}_{\appr}$ \eqref{second approximation of V bar}, we obtain an approximation of the Hamiltonian at order $O(\mu^2\epsilon)$, denoted $H_{\appr}$
\begin{align}\label{First approximated hamiltonian}
    H_{\appr} = \frac{1}{2} \int_{\mathbb{R}^d} \zeta^2 \dX + \frac{1}{2} \int_{\mathbb{R}^d} h \nabla \psi \cdot \nabla \psi \dX
    + \frac{\mu}{6} \int_{\mathbb{R}^d} \nabla(h^3\mathrm{F}_2 [\nabla\cdot\nabla\psi]) \cdot \nabla\psi \dX.
\end{align}
Now let's differentiate this approximated Hamiltonian in the sense of functional derivatives with respect to $\psi$ and $\zeta$, we get 
\begin{align}
    \begin{cases} 
        \delta_{\psi}H_{\appr}=-\nabla\cdot(h\nabla\psi)-\frac{\mu}{6}(\Delta(h^3\mathrm{F}_2[\Delta\psi]) + \Delta(\mathrm{F}_2[h^3\Delta\psi]),\\
        \delta_{\zeta}H_{\appr} = \zeta + \frac{\epsilon}{2}|\nabla\psi|^2 -\frac{\mu\epsilon}{2}h^2\mathrm{F}_2[\Delta\psi]\Delta\psi.
    \end{cases}
\end{align}
We can now write down the Hamilton equations on the approximated Hamiltonian $H_{\appr}$  
\begin{align}\label{first full dispersion model (in psi)}
    \begin{cases} \partial_t \zeta = -\nabla\cdot(h\nabla\psi) - \frac{\mu}{6}\big(\Delta(\mathrm{F}_2[h^3\Delta\psi]) + \Delta(h^3\mathrm{F}_2[\Delta\psi])\big), \\
    \partial_t \psi =  -\zeta - \frac{\epsilon}{2}|\nabla\psi|^2 +\frac{\mu\epsilon}{2}h^2\mathrm{F}_2[\Delta\psi]\Delta\psi.
    \end{cases}
\end{align}

\begin{myrem}
\begin{itemize}
    \item This system is the full dispersion equivalent of a Green-Naghdi system with variables $(\zeta,\psi)$ (set $\mathrm{F}_2 = I_d$ to get the latter), see \cite{CraigGroves94} and \cite{Whitham67}. This last one is never studied because of its ill-posedness at the linear level. But for \eqref{first full dispersion model (in psi)} the ill-posedness is not clear. Indeed, by construction when linearizing this system around the rest state, we obtain the same system as the linearized water waves equations, that is
        \begin{align*}
            \begin{cases}
                \partial_t\zeta + \mathrm{F}_1\Delta\psi = 0, \\
                \partial_t\psi + \zeta = 0,
            \end{cases}    
        \end{align*}
    which is well-posed in Sobolev spaces.
    \item The system \eqref{first full dispersion model (in psi)} is Hamiltonian by construction. Hence smooth solutions preserve energy $H_{\appr}$ in addition to mass $\int_{\mathbb{R}^d}\zeta$ and momentum $\int_{\mathbb{R}^d} \zeta\nabla\psi$. 
\end{itemize}
\end{myrem}
\label{Formal Derivation}

\subsection{Consistency with respect to the water waves system}
We now prove proposition \ref{Proposition consistency first full dispersion GN system}, on the consistency at order $O(\mu^2\epsilon)$ with respect to the water waves system (see definition \ref{Definition consistency}) of the first full dispersion system of Green-Naghdi type derived in the previous subsection \eqref{first full dispersion model (in psi)}. I recall the proposition here.
\begin{mypp}
    Let $\mathrm{F}_1$ and $\mathrm{F}_2$ be the Fourier multipliers defined in proposition \ref{Proposition Approximation Vbar estimates}. The water waves equations are consistent at order $O(\mu^2\epsilon)$ in the shallow water regime $\mathcal{A}$ (see definition \ref{Definition consistency}) with the following full dispersion Hamiltonian Green-Naghdi system 
    \begin{align}
    \begin{cases} \partial_t \zeta +\nabla\cdot(h\nabla\psi) + \frac{\mu}{6}\big(\Delta(\mathrm{F}_2[h^3\Delta\psi]) + \Delta(h^3\mathrm{F}_2[\Delta\psi])\big)=0, \\
    \partial_t \psi  +\zeta + \frac{\epsilon}{2}|\nabla\psi|^2 -\frac{\mu\epsilon}{2}h^2\mathrm{F}_2[\Delta\psi]\Delta\psi=0,
    \end{cases}
\end{align}
with $n=4$.
\end{mypp}

\begin{proof}
Let $\zeta$ and $\psi$ be the solutions of the water waves system \eqref{WaterWavesEquations}. Using the notations of definition \ref{Definition consistency} we have in our case 
\begin{align*}
    \begin{cases}
        \mathcal{N}_{(A)}^1(\zeta,\psi) := \nabla\cdot(h\nabla\psi) + \frac{\mu}{6}\big(\Delta(\mathrm{F}_2[h^3\Delta\psi]) + \Delta(h^3\mathrm{F}_2[\Delta\psi])\big), \\
        \mathcal{N}_{(A)}^2(\zeta,\psi) := \zeta + \frac{\epsilon}{2}|\nabla\psi|^2 - \frac{\mu\epsilon}{2}h^2\mathrm{F}_2[\Delta\psi]\Delta\psi.
    \end{cases}
\end{align*}
So we need to prove 
\begin{align}\label{Consistency estimate first system}
    \begin{cases}
        |\partial_t\zeta + \nabla\cdot(h\nabla\psi) + \frac{\mu}{6}\big(\Delta(\mathrm{F}_2[h^3\Delta\psi]) + \Delta(h^3\mathrm{F}_2[\Delta\psi])\big)|_{H^s} \leq \mu^2\epsilon N(s+4),\\
        |\partial_t\psi + \zeta + \frac{\epsilon}{2}|\nabla\psi|^2 - \frac{\mu\epsilon}{2}h^2\mathrm{F}_2[\Delta\psi]\Delta\psi|_{H^s} \leq \mu^2\epsilon N(s+4). 
    \end{cases}
\end{align}

\indent

\noindent \underline{Step 1}: Let's prove the first estimate of \eqref{Consistency estimate first system}.\\
Using the definition of the Dirichlet-to-Neumann operator $\mathcal{G}^{\mu}$ in term of the vertically averaged horizontal velocity $\overline{V}$, we know that the water waves solutions $(\zeta,\psi)$ satisfy 
\begin{align*}
    \partial_t \zeta + \nabla \cdot (h\overline{V}) = 0.
\end{align*}
Moreover we found an approximation of $\overline{V}$ of order $O(\mu^2\epsilon)$. I recall it here
\begin{align*}
    \widetilde{\overline{V}}_{\appr} = \nabla \psi + \frac{\mu}{3h}\nabla(h^3\mathrm{F}_2\Delta\psi)
\end{align*}
where $\mathrm{F}_2$ is a Fourier multiplier defined as $\mathrm{F}_2 = \frac{3}{\mu|D|^2}(1 - \frac{\tanh{(\sqrt{\mu}|D|)}}{\sqrt{\mu}|D|})$. For which we got the following estimations (see proposition \ref{Proposition Approximation Vbar estimates with VbarApprox})
\begin{align*}
    |\overline{V}-\widetilde{\overline{V}}_{\appr}|_{H^s} \leq \mu^2 \epsilon M(s+3)|\nabla\psi|_{H^{s+4}}.
\end{align*}
So we have 
\begin{align*}
    &|\partial_t\zeta +\nabla\cdot(h\widetilde{\overline{V}}_{\appr})|_{H^s}
    \leq |\partial_t\zeta +\nabla\cdot(h\overline{V})|_{H^s}+|\nabla\cdot(h(\overline{V}-\widetilde{\overline{V}}_{\appr}))|_{H^s} \\
    \leq &M(s+1)|\overline{V}-\widetilde{\overline{V}}_{\appr}|_{H^{s+1}}
    \leq \mu^2\epsilon M(s+4)|\nabla\psi|_{H^{s+5}}.
\end{align*}
Hence
\begin{align*}
    |\partial_t\zeta + \nabla\cdot(h\nabla\psi)+\frac{\mu}{3}\Delta(h^3\mathrm{F}_2[\Delta\psi])|_{H^s} \leq \mu^2\epsilon M(s+4)|\nabla\psi|_{H^{s+5}}.
\end{align*}
To prove the first estimation in \eqref{Consistency estimate first system} it only remains to prove that there exists $k, l \in \mathbb{N}$ with $k, l \leq 4$ such that
\begin{align*}
    |\Delta(h^3\mathrm{F}_2[\Delta\psi])-(\frac{1}{2}\Delta(h^3\mathrm{F}_2[\Delta\psi])+\frac{1}{2}\Delta(\mathrm{F}_2[h^3\Delta\psi]))|_{H^s} \leq \mu\epsilon M(s+k)|\nabla\psi|_{H^{s+l}}.
\end{align*}
Seeing that
\begin{align*}
    &\Delta(h^3\mathrm{F}_2[\Delta\psi])-(\frac{1}{2}\Delta(h^3\mathrm{F}_2[\Delta\psi])+\frac{1}{2}\Delta(\mathrm{F}_2[h^3\Delta\psi])) \\
    = &\frac{1}{2}\Delta((h^3-1)(\mathrm{F}_2-1)[\Delta\psi]) - \frac{1}{2}\Delta((\mathrm{F}_2-1)[(h^3-1)\Delta\psi]),
\end{align*}
we only need to use the estimates on $\mathrm{F}_2$ in proposition \ref{FourierMultiplierApproximations}, the product estimates \ref{product estimate} and the fact that $|h^3-1|_{H^{s+4}}\leq \epsilon M(s+4)$ to get
\begin{align*}
    &|\Delta((h-1)^3(\mathrm{F}_2-1)[\Delta\psi]) - \Delta((\mathrm{F}_2-1)[(h^3-1)\Delta\psi])|_{H^s} \\
    \leq &|\Delta((h-1)^3(\mathrm{F}_2-1)[\Delta\psi])|_{H^s} + |\Delta((\mathrm{F}_2-1)[(h^3-1)\Delta\psi])|_{H^s} \leq \mu\epsilon M(s+4)|\nabla\psi|_{H^{s+5}}.
\end{align*}
So
\begin{align*}
    |\partial_t\zeta + \nabla\cdot(h\nabla\psi) + \frac{\mu}{6}\big(\Delta(\mathrm{F}_2[h^3\Delta\psi]) + \Delta(h^3\mathrm{F}_2[\Delta\psi])\big)|_{H^s} \leq \mu^2\epsilon M(s+4) |\nabla\psi|_{H^{s+5}}.
\end{align*}

\indent

\noindent \underline{Step 2}: We now prove the second estimate of \eqref{Consistency estimate first system}.\\
We know that the solutions of the water waves system $(\zeta,\psi)$ satisfy
\begin{align*}
    \partial_t \psi + \zeta + \frac{\epsilon}{2}|\nabla\psi|^2-\frac{\mu\epsilon}{2}\frac{(\frac{1}{\mu}\mathcal{G}^{\mu}\psi + \epsilon\mu\nabla\zeta\cdot\nabla\psi)^2}{1+\epsilon^2\mu|\nabla\zeta|^2} = 0.
\end{align*}
\noindent Using quotient estimates \ref{Quotient estimate}, product estimates \ref{product estimate} and proposition \ref{Theorem 3.15} we get
\begin{align*}
    &|\frac{(\frac{1}{\mu}\mathcal{G}^{\mu}\psi + \epsilon\nabla\zeta\cdot\nabla\psi)^2}{1+\epsilon^2\mu|\nabla\zeta|^2}-(\frac{1}{\mu}\mathcal{G}^{\mu}\psi + \epsilon\nabla\zeta\cdot\nabla\psi)^2|_{H^s} \\
    = &|\frac{\mu\epsilon^2|\nabla\zeta|^2(\frac{1}{\mu}\mathcal{G}^{\mu}\psi + \epsilon\nabla\zeta\cdot\nabla\psi)^2}{1+\epsilon^2\mu|\nabla\zeta|^2}|_{H^s} \\
    \leq &\mu\epsilon^2 M(s+1) |(\frac{1}{\mu}\mathcal{G}^{\mu}\psi + \epsilon\nabla\zeta\cdot\nabla\psi)^2|_{H^s} \\
    \leq &\mu\epsilon^2 M(s+1)|\frac{1}{\mu}\mathcal{G}^{\mu}\psi + \epsilon\nabla\zeta\cdot\nabla\psi|_{H^{s+2}}^2 \\
    \leq &\mu\epsilon^2 M(s+1)(|\frac{1}{\mu}\mathcal{G}^{\mu}\psi|_{H^{s+2}} + \epsilon |\zeta|_{H^{s+3}}|\nabla\psi|_{H^{s+2}})^2\\
    \leq &\mu\epsilon^2 M(s+1) (M(s+4)|\nabla\psi|_{H^{s+3}})^2\\
    \leq &\mu\epsilon^2 C(M(s+4),|\nabla\psi|_{H^{s+3}}).
\end{align*}
So up to $O(\mu^2\epsilon)$ terms, we can replace the second equation of \eqref{WaterWavesEquations} by a simpler one, i.e. there exists $R_2 \in H^{s}$ such that $|R_2|_{H^s} \leq C(M(s+5),|\nabla\psi|_{H^{s+5}})$ and
\begin{align}\label{First system second equation simplification 1}
    \partial_t \psi + \zeta + \frac{\epsilon}{2}|\nabla\psi|^2-\frac{\mu\epsilon}{2}(-\nabla\cdot(h\overline{V}) + \epsilon\nabla\zeta\cdot\nabla\psi)^2= \mu^2\epsilon R_2.
\end{align}
Now we need a proposition proved in \cite{WWP} (see proposition 3.37 and remark 3.40).

\begin{mypp}\label{Vbar moins nabla psi estimation}
    Let $t_0 > d/2$, $s \geq 0$, and $\zeta \in H^{\max{(t_0+1,s+2)}}(\mathbb{R}^d)$ be such that \eqref{FundamentalHypothesis} is satisfied. Let $\psi \in \dot{H}^{s+2}(\mathbb{R}^d)$, and $\overline{V}$ be as in \eqref{VerticallyAveragedHorizontalVelocity}. Then we have the following error estimates
    \begin{align*}
        \begin{cases}
            |\overline{V}| \leq M(s+2)|\nabla\psi|_{H^{s+2}},\\
            |\overline{V}-\nabla\psi|_{H^s} \leq \mu M(s+2) |\nabla\psi|_{H^{s+2}}.
        \end{cases}
    \end{align*}
\end{mypp}
\noindent Using proposition \ref{Vbar moins nabla psi estimation} and product estimates \ref{product estimate} we have
\begin{align*}
    &|(-\nabla\cdot(h\overline{V}) + \epsilon\nabla\zeta\cdot\nabla\psi)^2 - (-\nabla\cdot(h\nabla\psi) + \epsilon\nabla\zeta\cdot\nabla\psi)^2|_{H^s} \\
    = & |(-\nabla\cdot(h(\overline{V} - \nabla\psi))(-\nabla\cdot(h\overline{V}) + 2\epsilon\nabla\zeta\cdot\nabla\psi -\nabla\cdot(h\nabla\psi)|_{H^{s}}\\
    \leq &|\nabla\cdot(h(\overline{V} - \nabla\psi)|_{H^{s+1}}|-\nabla\cdot(h\overline{V}) + 2\epsilon\nabla\zeta\cdot\nabla\psi -\nabla\cdot(h\nabla\psi)|_{H^{s+1}}\\
    \leq &\mu C(M(s+4),|\nabla\psi|_{H^{s+4}}).
\end{align*}
Hence up to $O(\mu^2\epsilon)$ terms, we can replace \eqref{First system second equation simplification 1} by a simpler one, i.e. there exists $R_2 \in H^{s}$ such that $|R_2|_{H^s} \leq C(M(s+4),|\nabla\psi|_{H^{s+4}})$ and
\begin{align}\label{First system second equation simplification 2}
     \partial_t \psi + \zeta + \frac{\epsilon}{2}|\nabla\psi|^2-\frac{\mu\epsilon}{2}
     h^2(\Delta\psi)^2= \mu^2\epsilon R_2.
\end{align}
Now it only remains to use the estimates on $F_2$ in proposition \ref{FourierMultiplierApproximations} to get
\begin{align*}
    |h^2\mathrm{F}_2[\Delta\psi]\Delta\psi - h^2(\Delta\psi)^2|_{H^{s}} \leq \mu C(M(s+2),|\nabla\psi|_{H^{s+3}})    
\end{align*}
So there exists $R_2 \in H^s$ such that $|R_2| \leq C(M(s+4),|\nabla\psi|_{H^{s+4}})$ and
\begin{align*}
    \partial_t \psi + \zeta + \frac{\epsilon}{2}|\nabla\psi|^2-\frac{\mu\epsilon}{2}
     h^2\mathrm{F}_2[\Delta\psi]\Delta\psi= \mu^2\epsilon R_2.
\end{align*}
Thus we proved the consistency of the water waves equations \eqref{WaterWavesEquations} at order $O(\mu^2\epsilon)$ in the shallow water regime with the system \eqref{first full dispersion model (in psi)} with $n=4$.
\end{proof}

\label{Consistency with respect to the water waves system}

\section{Derivation and consistency of a second full dispersion Green-Naghdi system}
\label{Derivation and consistency of a second full dispersion Green-Naghdi system}

\subsection{Formal Derivation}
In this subsection we explain formally how to obtain the second full dispersion Green-Naghdi system \eqref{SecondFullDispersionModel(InVbar)} using \eqref{approximation of phi in terms of V bar}. But first let's symmetrize \eqref{approximation of phi in terms of V bar}. It yields 
\begin{align}\label{approximation of grad psi in terms of V bar symmetrised version}
    h\nabla\psi = h\overline{V}-\frac{\mu}{6}(\nabla(h^3\mathrm{F}_3 [\nabla\cdot \overline{V}])+\nabla(\mathrm{F}_3[h^3\nabla\cdot \overline{V}])+ O(\mu^2\epsilon).
\end{align}
We will define two operators $T[h]V = -\frac{1}{6h}(\nabla(h^3\mathrm{F}_3[\nabla \cdot V]) + \nabla(\mathrm{F}_3[h^3\nabla \cdot V]))$ and $I[h]V = h(V + \mu T[h]V)$, such that the previous approximation of $\nabla \psi$ can be written:
\begin{equation}\label{Approximation of h nabla psi with Ih}
    h\nabla\psi = I[h]\overline{V} + O(\mu^2\epsilon)
\end{equation}

\begin{myrem}
The choice of the symmetrization is arbitrary. Here we use the same as for the first full dispersion Green-Naghdi system \eqref{first full dispersion model (in psi)}, for which the symmetrization naturaly comes up when asking the system to be Hamiltonian. See \eqref{Full dispersion Green Nagdhi model article Vincent} for another full-dispersion Green-Naghdi system with a kind of symmetrization already appearing in the litterature \cite{DucheneIsrawiTalhouk16}.
\end{myrem}
\noindent We suppose $I[h]$ invertible and do formally all the computations with this assumption.\\
Let H be the Hamiltonian of the Zakharov/Craig-Sulem formulation:
\begin{align*}
    H &= \frac{1}{2} \int_{\mathbb{R}^d} \zeta^2 dX + \frac{1}{2} \int_{\mathbb{R}^d} h\overline{V} \cdot \nabla\psi dX.
\end{align*}
Using \eqref{Approximation of h nabla psi with Ih} an approximated Hamiltonian would be:
\begin{equation}\label{Second approximated hamiltonian}
    H_{\appr} = \frac{1}{2} \int_{\mathbb{R}^d} \zeta^2 dX + \frac{1}{2} \int_{\mathbb{R}^d} h I[h]^{-1}[h\nabla\psi]\cdot\nabla\psi dX.
\end{equation}
Let's differentiate this Hamiltonian in the sense of functional derivative.\\ 
Through some computations we first get the derivative of $H_{\appr}$ with respect to $\psi$
\begin{equation*}
    \delta_{\psi}H = -\nabla \cdot(h I[h]^{-1}[h\nabla\psi]).
\end{equation*}
To compute the derivative of $H_{\appr}$ with respect to $\zeta$ we use the formula
\begin{align}\label{Inverse formula Ih}
    I[h]\partial_{\zeta}(I[h]^{-1})I[h]V=-\partial_{\zeta}I[h]V.
\end{align}
From \eqref{Inverse formula Ih} and the fact that $I[h]^{-1}$ is a symmetric operator, we can compute the derivative in $\zeta$ of the second term of $H_{\appr}$. We get 
\begin{align*}
    \delta_{\zeta}H = &\zeta + \frac{\epsilon}{2} I[h]^{-1}[h\nabla\psi]\cdot \nabla \psi
    - \frac{\epsilon}{2} I[h]^{-1}[h\nabla\psi] \cdot I[h]^{-1}[h\nabla\psi] \\
    + &\frac{\mu\epsilon}{4}\nabla \cdot(I[h]^{-1}[h\nabla \psi])h^2\mathrm{F}_3[\nabla \cdot (I[h]^{-1}[h\nabla\psi])] 
    + \frac{\mu\epsilon}{4} \nabla \cdot(\mathrm{F}_3[I[h]^{-1}[h\nabla\psi]])h^2\nabla \cdot (I[h]^{-1}[h\nabla \psi]) \\
    + & \frac{\epsilon}{2} I[h]^{-1}[h\nabla\psi]\cdot \nabla\psi.
\end{align*}
Then we define $\overline{V}_{\appr} \colonequals (h(Id+\mu T[h]))^{-1} [h\nabla \psi] = I[h]^{-1}[h\nabla\psi]$. This new quantity approximates $\overline{V}$ at order $O(\mu^2\epsilon)$, i.e. $\overline{V}_{\appr} =  \overline{V}+ O(\mu^2\epsilon)$. The two functional derivatives $\delta_{\psi}H_{\appr}$ and $\delta_{\zeta}H_{\appr}$ allow us to obtain the equations of movement in $\zeta$ and $\overline{V}_{\appr}$. \\
First the conservation of mass:
\begin{align*}
    \partial_t \zeta = \delta_{\psi}H \iff \partial_t \zeta &= -\nabla\cdot (hI[h]^{-1}[h\nabla \psi]) \\
    &= -\nabla \cdot(h\overline{V}_{\appr}).
\end{align*}
And next the conservation of momentum:
\begin{align*}
    &\partial_t \psi = -\delta_{\zeta}H \\
    \iff &\partial_t \psi = -\zeta -\frac{\epsilon}{2}\nabla \psi \cdot I[h]^{-1}[h\nabla\psi] 
    + \frac{\epsilon}{2} I[h]^{-1}[h\nabla \psi]\cdot I[h]^{-1}[h\nabla\psi] \\
    + & \frac{\mu\epsilon}{4} \nabla \cdot (I[h]^{-1}[h\nabla \psi]) h^2 \mathrm{F}_3[\nabla \cdot(I[h]^{-1}[h\nabla\psi])]\\
    +  &\frac{\mu\epsilon}{4} \nabla \cdot (\mathrm{F}_3[I[h]^{-1}[h\nabla\psi]])h^2\nabla \cdot(I[h]^{-1}[h\nabla\psi]) 
    - \frac{\epsilon}{2} I[h]^{-1}[h\nabla\psi]\cdot \nabla \psi.
\end{align*}
Then applying $\nabla$, we obtain the conservation of momentum
\begin{align*}
    \partial_t((Id+\mu T[h])\overline{V}_{\appr}) = &-\nabla \zeta -\frac{\epsilon}{2} \nabla(|\overline{V}_{\appr}|^2) + \frac{\mu\epsilon}{6}\nabla(\frac{1}{h}\overline{V}_{\appr}\cdot \nabla(h^3\mathrm{F}_3[\nabla\cdot\overline{V}_{\appr}]+\mathrm{F}_3[h^3\nabla\cdot\overline{V}_{\appr}])) \\
    + &\frac{\mu\epsilon}{2}\nabla(h^2\mathrm{F}_3[\nabla\cdot\overline{V}_{\appr}]\nabla\cdot\overline{V}_{\appr}).
\end{align*}
Thus we get the second full dispersion Green-Naghdi model \eqref{SecondFullDispersionModel(InVbar)} that we recall here
\begin{align*}
    \begin{cases}
        \partial_t \zeta + \nabla \cdot(h\overline{V}) = 0, \\
        \partial_t((Id+\mu T[h])\overline{V}) +\nabla \zeta &+\frac{\epsilon}{2} \nabla(|\overline{V}|^2) - \frac{\mu\epsilon}{6}\nabla(\frac{1}{h}\overline{V}\cdot \nabla(h^3\mathrm{F}_3[\nabla\cdot\overline{V}]+\mathrm{F}_3[h^3\nabla\cdot\overline{V}])) \\
        &- \frac{\mu\epsilon}{2}\nabla(h^2\mathrm{F}_3[\nabla\cdot\overline{V}]\nabla\cdot\overline{V}) =0.
    \end{cases}
\end{align*}

\label{Formal Derivation 2}

\subsection{Consistency with respect to the water waves equations}
Taking the same notations and definitions as in subsection 3.2, we prove here proposition \ref{Proposition consistency second Full dispersion Gn System}. For the sake of clarity we recall it here.
\begin{mypp}\label{Consistency second full dispersion GN system section 4}
    Let $\mathrm{F}_3$ be the Fourier multiplier defined in proposition \ref{Proposition Approximation Vbar estimates}. Let $T[h]V \colonequals -\frac{1}{6h}(\nabla(h^3\mathrm{F}_3[\nabla \cdot V]) + \nabla(\mathrm{F}_3[h^3\nabla \cdot V]))$. The water waves equations are consistent at order $O(\mu^2\epsilon)$ in the shallow water regime $\mathcal{A}$ with
    \begin{align}\label{SecondFullDispersionModel(InVbar)2}
    \begin{cases}
        \partial_t \zeta + \nabla \cdot(h\overline{V})=0, \\
        \partial_t((Id+\mu T[h])\overline{V}) +\nabla \zeta &+\frac{\epsilon}{2} \nabla(|\overline{V}|^2) - \frac{\mu\epsilon}{6}\nabla(\frac{1}{h}\overline{V}\cdot \nabla(h^3\mathrm{F}_3[\nabla\cdot\overline{V}]+\mathrm{F}_3[h^3\nabla\cdot\overline{V}])) \\
        &- \frac{\mu\epsilon}{2}\nabla(h^2\mathrm{F}_3[\nabla\cdot\overline{V}]\nabla\cdot\overline{V}) =0.
    \end{cases}
\end{align}
with $n=6$.
\end{mypp}

\begin{myrem}
    The first equation of $\eqref{SecondFullDispersionModel(InVbar)2}$ is exact. There's nothing to prove for this one. All the work is on the second equation.
\end{myrem}

\begin{proof}
Let $\zeta$ and $\psi$ (so $\overline{V}$ through \eqref{VerticallyAveragedHorizontalVelocity}) be the solutions of the water waves system \eqref{WaterWavesEquations}. Using the notations of definition \ref{Definition consistency} we have in our case 
\begin{multline}\label{operator second full dispersion Green Naghdi system}
    \mathcal{N}_{(A')}^3(\zeta,\overline{V}) := \mu\partial_t(T[h]\overline{V}) +\nabla \zeta +\frac{\epsilon}{2} \nabla(|\overline{V}|^2) - \frac{\mu\epsilon}{6}\nabla(\frac{1}{h}\overline{V}\cdot \nabla(h^3\mathrm{F}_3[\nabla\cdot\overline{V}]+\mathrm{F}_3[h^3\nabla\cdot\overline{V}]))\\
    - \frac{\mu\epsilon}{2}\nabla(h^2\mathrm{F}_3[\nabla\cdot\overline{V}]\nabla\cdot\overline{V}).
\end{multline}
So we need to prove 
\begin{multline}\label{Consistency estimate second system}
        |\partial_t\overline{V} + \nabla\cdot(h\nabla\psi) + \mu\partial_t(T[h]\overline{V}) +\nabla \zeta +\frac{\epsilon}{2} \nabla(|\overline{V}|^2) - \frac{\mu\epsilon}{6}\nabla(\frac{1}{h}\overline{V}\cdot \nabla(h^3\mathrm{F}_3[\nabla\cdot\overline{V}]+\mathrm{F}_3[h^3\nabla\cdot\overline{V}]))\\ - \frac{\mu\epsilon}{2}\nabla(h^2\mathrm{F}_3[\nabla\cdot\overline{V}]\nabla\cdot\overline{V})|_{H^s} \leq \mu^2\epsilon N(s+7)
\end{multline}

\indent 

\noindent \underline{Step 1}: Here we focus on the terms of the second equation of \eqref{SecondFullDispersionModel(InVbar)2} which are not differentiated in time and prove that there exists $R_3 \in H^s(\mathbb{R}^d)$ such that $|R_3|_{H^s} \leq N(s+6)$ and 
\begin{multline}\label{What we want to prove step 1}
    \partial_t \nabla\psi + \nabla\zeta + \frac{\epsilon}{2}\nabla(|\overline{V}|^2)-\frac{\mu\epsilon}{6}\nabla(\frac{1}{h}\overline{V}\cdot\nabla(h^3\mathrm{F}_3[\nabla\cdot\overline{V}]+\mathrm{F}_3[h^3\nabla\cdot\overline{V}])) \\
    - \frac{\mu\epsilon}{2}\nabla(h^2\mathrm{F}_3[\nabla\cdot\overline{V}]\nabla\cdot\overline{V}) = \mu^2\epsilon R_3
\end{multline}
Taking $(\zeta,\psi)$ solutions of the water waves system \eqref{WaterWavesEquations} we proved in subsection 3.2 (see \eqref{First system second equation simplification 2}) that there exists a remainder $R_2 \in H^s(\mathbb{R}^d)$ such that $|R_2|_{H^s} \leq N(s+4)$ and
\begin{align}\label{First full dispersion GN system simplified}
    \partial_t \psi + \zeta + \frac{\epsilon}{2}|\nabla\psi|^2 - \frac{\mu\epsilon}{2}h^2(\Delta\psi)^2 = \mu^2\epsilon R_2.
\end{align}
Using proposition \ref{FourierMultiplierApproximations} for $\mathrm{F}_3$, also taking the gradient of equation \eqref{First full dispersion GN system simplified}, using proposition \ref{Vbar moins nabla psi estimation}, product estimate \ref{product estimate} and the boundedness of $\mathrm{F}_3$ we get the existence of $R_3 \in H^s(\mathbb{R}^d)$ such that $|R_3|_{H^s} \leq N(s+5)$ and
\begin{align*}
     \partial_t \nabla\psi + \nabla\zeta + \frac{\epsilon}{2}\nabla(|\nabla\psi|^2) - \frac{\mu\epsilon}{2}\nabla(h^2\mathrm{F}_3[\nabla\cdot\overline{V}]\nabla\cdot\overline{V}) = \mu^2\epsilon R_3.
\end{align*}

\noindent In order to go further we need a symmetrized version of \eqref{approximation of nabla psi by Vbar}.
\begin{mypp}\label{proposition symmetrized approximation of nabla psi by Vbar}
    Let $s \geq 0$, and $\zeta \in H^{s+5}(\mathbb{R}^d)$ be such that \eqref{FundamentalHypothesis} is satisfied. Let $\psi \in \dot{H}^{s+5}(\mathbb{R}^d)$, and $\overline{V}$ be as in \eqref{VerticallyAveragedHorizontalVelocity}. Let also $\mathrm{F}_3$ be the Fourier multiplier defined in proposition \ref{Proposition Approximation Vbar estimates}.\\
    The following estimate hold:
    \begin{align*}
        |\overline{V}-\nabla\psi - \frac{\mu}{6h}\nabla(h^3\mathrm{F}_3[\nabla\cdot\overline{V}]+\mathrm{F}_3[h^3\nabla\cdot\overline{V}])|_{H^s} \leq \mu^2\epsilon M(s+5)|\nabla\psi|_{H^{s+5}}.
    \end{align*}
\end{mypp}
\begin{proof}
    Using \eqref{approximation of nabla psi by Vbar} and quotient estimates \ref{Quotient estimate} we get 
    \begin{align*}
        &|\overline{V}-\nabla\psi - \frac{\mu}{6h}\nabla(h^3\mathrm{F}_3[\nabla\cdot\overline{V}]+\mathrm{F}_3[h^3\nabla\cdot\overline{V}])|_{H^s}\\
        \leq &|\overline{V}-\nabla\psi -\frac{\mu}{3h}\nabla(h^3 \mathrm{F}_3 \nabla \cdot \overline{V})|_{H^s} + |\frac{\mu}{3h}\nabla(h^3 \mathrm{F}_3 \nabla \cdot \overline{V})-\frac{\mu}{6h}\nabla(h^3\mathrm{F}_3[\nabla\cdot\overline{V}]+\mathrm{F}_3[h^3\nabla\cdot\overline{V}])|_{H^s}\\
        \leq &\mu^2\epsilon M(s+4) |\nabla\psi|_{H^{s+4}} + \mu M_0 |(h^3-1)(\mathrm{F}_3-1)[\nabla\cdot\overline{V}] - (\mathrm{F}_3-1)[(h^3-1)\nabla\cdot\overline{V}]|_{H^{s+1}}\\
        \leq &\mu^2\epsilon M(s+4) |\nabla\psi|_{H^{s+4}} + \mu M_0 |[\mathrm{F}_3-1;h^3-1]\nabla\cdot\overline{V}|_{H^{s+1}}.
    \end{align*}
    Moreover using commutator estimates \ref{Commutator estimates} with $\mathrm{F}_3-1$ Fourier multiplier of order $2$ and $\mathcal{N}^2(\mathrm{F}_3-1) \lesssim \mu$ (see definition \ref{definition order of a Fourier multiplier}). Using also proposition \ref{FourierMultiplierApproximations} for $\mathrm{F}_3$ and proposition \ref{Vbar moins nabla psi estimation}, we obtain
    \begin{align*}
        &|[\mathrm{F}_3-1;h^3-1]\nabla\cdot\overline{V}|_{H^{s+1}}\\
        \leq &|[(\mathrm{F}_3-1)\Lambda^{s+1},h^3-1]\nabla\cdot\overline{V}|_2 + |[\Lambda^{s+1},h^3-1](\mathrm{F}_3-1)[\nabla\cdot\overline{V}]|_2\\
        \leq &\mu\epsilon|\frac{h^3-1}{\epsilon}|_{H^{s+3}}|\nabla\cdot\overline{V}|_{H^{s+2}} + \mu\epsilon|\frac{h^3-1}{\epsilon}|_{H^{\max{(t_0+1,s+1)}}}|(\mathrm{F}_3-1)\nabla\cdot\overline{V}|_{H^{s}}\\
        \leq &\mu\epsilon M(s+3)|\overline{V}|_{H^{s+3}} \leq \mu\epsilon M(s+5)|\nabla\psi|_{H^{s+5}}.
    \end{align*}
    Here $t_0$ is a real number larger than $d/2$, see remark \ref{remark on t0}.   
\end{proof}
\noindent Having in mind this symmetrized approximation of $\nabla\psi$ by $\overline{V}$ we estimate
\begin{align*}
    &|\frac{\epsilon}{2}\nabla(|\nabla\psi|^2) - (\frac{\epsilon}{2}\nabla(|\overline{V}|^2)-\frac{\mu\epsilon}{6}\nabla(\frac{1}{h}\overline{V}\cdot\nabla(h^3\mathrm{F}_3[\nabla\cdot\overline{V}]+\mathrm{F}_3[h^3\nabla\cdot\overline{V}])))| \\
    \leq &|\frac{\epsilon}{2}\nabla(|\nabla\psi|^2) - \frac{\epsilon}{2}\nabla(|\overline{V}-\frac{\mu}{6h}\nabla(h^3\mathrm{F}_3[\nabla\cdot\overline{V}]+\mathrm{F}_3[h^3\nabla\cdot\overline{V}])|^2)|_{H^s} \\
    + &\frac{\mu^2\epsilon}{2}|\nabla(\frac{1}{6^2 h^2}|\nabla(h^3\mathrm{F}_3[\nabla\cdot\overline{V}]+\mathrm{F}_3[h^3\nabla\cdot\overline{V}])|^2)|_{H^s}\\
    \colonequals &I_1 + I_2.
\end{align*}
Remark that the second term of the last inequality $I_2$ is of order $O(\mu^2\epsilon)$. We will get an estimation of it a bit later.\\
For now let's use proposition \ref{proposition symmetrized approximation of nabla psi by Vbar}, product estimates \ref{product estimate}, and quotient estimates \ref{Quotient estimate} to estimate the first term $I_1$:
\begin{align*}
    &|\frac{\epsilon}{2}\nabla(|\nabla\psi|^2) - \frac{\epsilon}{2}\nabla(|\overline{V}-\frac{\mu}{6h}\nabla(h^3\mathrm{F}_3[\nabla\cdot\overline{V}]+\mathrm{F}_3[h^3\nabla\cdot\overline{V}])|^2)|_{H^s} \\
    \leq &\frac{\epsilon}{2}||\nabla\psi|^2 - |\overline{V}-\frac{\mu}{6h}\nabla(h^3\mathrm{F}_3[\nabla\cdot\overline{V}]+\mathrm{F}_3[h^3\nabla\cdot\overline{V}])|^2|_{H^{s+1}} \\
    \leq &\frac{\epsilon}{2}|\nabla\psi-(\overline{V}-\frac{\mu}{6h}\nabla(h^3\mathrm{F}_3[\nabla\cdot\overline{V}]+\mathrm{F}_3[h^3\nabla\cdot\overline{V}]))|_{H^{s+1}}\\
    \times &|\nabla\psi+\overline{V}-\frac{\mu}{6h}\nabla(h^3\mathrm{F}_3[\nabla\cdot\overline{V}]+\mathrm{F}_3[h^3\nabla\cdot\overline{V}])|_{H^{s+2}} \\
    \leq &\mu^2\epsilon^2 M(s+6)|\nabla\psi|_{H^{s+6}}  |\nabla\psi+\overline{V}-\frac{\mu}{6h}\nabla(h^3\mathrm{F}_3[\nabla\cdot\overline{V}]+\mathrm{F}_3[h^3\nabla\cdot\overline{V}])|_{H^{s+2}}.
\end{align*}
So to finish the second step of this proof we just need to estimate the two quantities allowing us to control $I_1$ and $I_2$:
\begin{align*}
    \begin{cases}
        |\nabla(\frac{1}{6^2 h^2}|\nabla(h^3\mathrm{F}_3[\nabla\cdot\overline{V}]+\mathrm{F}_3[h^3\nabla\cdot\overline{V}])|^2)|_{H^s}, \\
        |\nabla\psi+\overline{V}-\frac{\mu}{6h}\nabla(h^3\mathrm{F}_3[\nabla\cdot\overline{V}]+\mathrm{F}_3[h^3\nabla\cdot\overline{V}])|_{H^{s+2}}.
    \end{cases}
\end{align*}
For the first one we can use quotient estimate \ref{Quotient estimate} and product estimates \ref{product estimate} to show that
\begin{align*}
    |\nabla(\frac{1}{6^2 h^2}|\nabla(h^3\mathrm{F}_3[\nabla\cdot\overline{V}]+\mathrm{F}_3[h^3\nabla\cdot\overline{V}])|^2)|_{H^s} \leq &M(s+1)||\nabla(h^3\mathrm{F}_3[\nabla\cdot\overline{V}]+\mathrm{F}_3[h^3\nabla\cdot\overline{V}])|^2|_{H^{s+1}} \\
    \leq &M(s+1)|\nabla(h^3\mathrm{F}_3[\nabla\cdot\overline{V}]+\mathrm{F}_3[h^3\nabla\cdot\overline{V}])|_{H^{s+2}}^2.
\end{align*}
And for the second one we can use proposition \ref{Vbar moins nabla psi estimation}, quotient estimate \ref{Quotient estimate} and proposition \ref{Vbar moins nabla psi estimation} to get
\begin{align*}
    &|\nabla\psi+\overline{V}-\frac{\mu}{6h}\nabla(h^3\mathrm{F}_3[\nabla\cdot\overline{V}]+\mathrm{F}_3[h^3\nabla\cdot\overline{V}])|_{H^{s+2}} \\
    \leq & M(s+4)|\nabla\psi|_{H^{s+4}} +\mu M(s+2)|\nabla(h^3\mathrm{F}_3[\nabla\cdot\overline{V}]+\mathrm{F}_3[h^3\nabla\cdot\overline{V}])|_{H^{s+2}}.
\end{align*}
Hence for both quantities it only remains to estimate $|\nabla(h^3\mathrm{F}_3[\nabla\cdot\overline{V}]+\mathrm{F}_3[h^3\nabla\cdot\overline{V}])|_{H^{s+2}}$. Let's do it.\\
Remark that the following inequality on $F_3$ (see proposition \ref{Proposition Approximation Vbar estimates}) holds:
\begin{align}\label{Inequality F3}
    |F_3(\xi)| = |\frac{3}{\mu|\xi|^2}(\frac{\sqrt{\mu}|\xi|}{\tanh{(\sqrt{\mu}|\xi|)}} - 1)| \leq \frac{1}{1+\frac{\sqrt{\mu}|\xi|}{3}}.
\end{align}
Using product estimates \ref{product estimate}, \eqref{Inequality F3} and proposition \ref{Vbar moins nabla psi estimation}, we get
\begin{align*}
    &|\nabla(h^3\mathrm{F}_3[\nabla\cdot\overline{V}]+\mathrm{F}_3[h^3\nabla\cdot\overline{V}])|_{H^{s+2}} \\
    \leq &M(s+3)|\mathrm{F}_3[\nabla\cdot\overline{V}]|_{H^{s+3}}+|\mathrm{F}_3[h^3\nabla\cdot\overline{V}])|_{H^{s+3}} \\
    \leq &M(s+3)|\overline{V}|_{H^{s+3}} \leq M(s+5)|\nabla\psi|_{H^{s+5}}.
\end{align*}
Thus 
\begin{align*}
    &|\frac{\epsilon}{2}\nabla(|\nabla\psi|^2) - (\frac{\epsilon}{2}\nabla(|\overline{V}|^2)-\frac{\mu\epsilon}{6}\nabla(\frac{1}{h}\overline{V}\cdot\nabla(h^3\mathrm{F}_3[\nabla\cdot\overline{V}]+\mathrm{F}_3[h^3\nabla\cdot\overline{V}])))|_{H^s} \\
    \leq &\mu^2\epsilon C(M(s+6),|\nabla\psi|_{H^{s+6}}).
\end{align*}
And we proved \eqref{What we want to prove step 1}.

To prove the consistency of the water waves equations \eqref{WaterWavesEquations} with respect to the second full dispersion Green-Naghdi system \eqref{SecondFullDispersionModel(InVbar)2} at order $O(\mu^2\epsilon)$ it only remains to focus on the term differentiated in time and show that there exists $k\in \mathbb{N}$ with $k \leq 7$ such that
\begin{align*}
    |\partial_t (\nabla\psi-(\overline{V}-\frac{\mu}{6h} \nabla(h^3\mathrm{F}_3[\nabla\cdot\overline{V}]+\mathrm{F}_3[h^3\nabla\cdot\overline{V}])))|_{H^s} \leq \mu^2 \epsilon N(s+k).
\end{align*}
What we will prove is in fact
\begin{align*}
    &|\partial_t (\nabla\psi-(\overline{V}-\frac{\mu}{3h} \nabla(h^3\mathrm{F}_3[\nabla\cdot\overline{V}])))|_{H^s} \\
    + &|\partial_t(\frac{\mu}{3h} \nabla(h^3\mathrm{F}_3[\nabla\cdot\overline{V}]) - \frac{\mu}{6h} \nabla(h^3\mathrm{F}_3[\nabla\cdot\overline{V}]+\mathrm{F}_3[h^3\nabla\cdot\overline{V}]))|_{H^s}
    \leq \mu^2 \epsilon N(s+7).
\end{align*}

\indent 

\noindent \underline{Step 2}: We estimate first
\begin{align}\label{What we want to estimate 1}
    |\partial_t (\nabla\psi-(\overline{V}-\frac{\mu}{3h} \nabla(h^3\mathrm{F}_3[\nabla\cdot\overline{V}])))|_{H^s}. 
\end{align}
In that end we denote $\widetilde{u}=\phi-\widetilde{\phi}_{\appr}$, where $\phi$ is defined in definition \ref{Definitions Fondamentales} and $\widetilde{\phi}_{\appr}$ in \eqref{Second approximation of phi}.\\
\underline{Step 2.1}: Here we find a control on $\widetilde{u}$ and prove
\begin{align}\label{control on u tilde step 2.1}
    ||\Lambda^s\nabla^{\mu}\partial_t \widetilde{u}||_2 \leq \mu^2\epsilon N(s+5).
\end{align}
By definition of $\phi$ and $\widetilde{\phi}_{\appr}$ (see definition \ref{Definitions Fondamentales}) we know that $\widetilde{u}$ solves an elliptic problem: 
\begin{align}\label{straightened laplacian of utilde}
    \begin{cases}
    h\nabla^{\mu}\cdot P(\Sigma_t)\nabla^{\mu}(\widetilde{u}) = -\mu^2\epsilon R,\\
    \widetilde{u}|_{z=0}=0, \ \ \partial_z \widetilde{u}|_{z=-1}=0,
    \end{cases}
\end{align}
where 
\begin{multline}\label{Expression du reste}
    R = \frac{1}{\mu^2\epsilon} \Big[ \mu(h^2-1)(\mathrm{F}_0-1)\Delta\psi + \mu\Delta((h^2-1)(\mathrm{F}_0-1)\psi) \\
    + \mu\epsilon A(\nabla,\partial_z) (\mathrm{F}_0-1)\psi + \mu\epsilon A(\nabla,\partial_z)(h^2-1)(\mathrm{F}_0-1)\psi \Big].
\end{multline}
In add using proposition \ref{FourierMultiplierApproximations} for $\mathrm{F}_0$, and product estimates \ref{product estimate}, we have the following estimation of the remainder $R$:
\begin{align}\label{Estimates on R}
    ||\Lambda^s R||_2 \leq M(s+2) |\nabla\psi|_{H^{s+3}}.
\end{align}
Moreover we can differentiate in time the elliptic equation in \eqref{straightened laplacian of utilde} as follow
\begin{align*}
    &\partial_t( \nabla^{\mu}\cdot P(\Sigma_t)\nabla^{\mu}(\widetilde{u}))= -\mu^2\epsilon \partial_t R \\
    \iff & \nabla^{\mu}\cdot \partial_t(P(\Sigma_t))\nabla^{\mu}\widetilde{u} + \nabla^{\mu}\cdot P(\Sigma_t)\nabla^{\mu} \partial_t \widetilde{u} = -\mu^2\epsilon \partial_t R.
\end{align*}
where here $R = \eqref{Expression du reste}/h$ (I use the same notation for both remainders, thanks to quotient estimates \ref{Quotient estimate}, the previous estimation holds).\\
It invites us to denote $v:=\partial_t \widetilde{u}$ and decompose it into $v \colonequals v_1 + v_2$ where $v_1$ satisfy one elliptic problem
\begin{align}\label{first elliptic problem for v_1}
    \begin{cases}
        \nabla^{\mu}\cdot P(\Sigma_t) \nabla^{\mu} v_1 = -\mu^2\epsilon \partial_t R, \\
        v_1|_{z=0}=0, \ \ \partial_z v_1|_{z=-1}=0,
    \end{cases}
\end{align}
and $v_2$ satisfy, for $g\colonequals \partial_t P(\Sigma_t) \nabla^{\mu} \widetilde{u}$ (see definition \ref{Definitions Fondamentales} for the expression of $P(\Sigma_t)$), a second elliptic problem
\begin{align}\label{Second elliptic problem for v2}
    \begin{cases}
        \nabla^{\mu}\cdot P(\Sigma_t) \nabla^{\mu} v_2 = -\nabla^{\mu}\cdot g, \\
        v_2|_{z=0}=0, \ \ v_2|_{z=-1}=-e_z\cdot g|_{z=-1}.
    \end{cases}
\end{align}
Thanks to the lemma \ref{lemma 3.43} we have a control on $v_1$ given by
\begin{align}\label{Estimates on v_1}
    ||\Lambda^s \nabla^{\mu}v_1||_2 \leq \mu^2\epsilon M(s+1) ||\Lambda^s\partial_t R||_2.
\end{align}
And having an explicit form of $R$ we can easily find an estimation of $\partial_t R$ using quotient estimates \ref{Quotient estimate} and product estimates \ref{product estimate}:
\begin{align}\label{Estimates on partial t R}
    ||\Lambda^s\partial_t R||_2 \leq &C(M(s+2),|\partial_t \zeta|_{H^{s+3}})|(\nabla\psi,\partial_t \nabla\psi)|_{H^{s+3}}.
\end{align}
Then using the water waves equations \eqref{WaterWavesEquations} we obtain estimates on the partial derivatives in time of $\zeta$ and $\nabla\psi$.

\begin{mylem}\label{Estimates on partial derivatives of zeta and psi}
    Let $s \geq 0$, and $\zeta \in H^{s+4}(\mathbb{R}^d)$ . Let $\psi \in \dot{H}^{s+4}(\mathbb{R}^d)$.\\
    The two estimations hold:
    \begin{align*}
        \begin{cases}
             |\partial_t \zeta|_{H^{s+2}} \leq N(s+4),\\
             |\partial_t \nabla\psi|_{H^{s+2}} \leq N(s+4).
        \end{cases}
    \end{align*}
\end{mylem}

\begin{proof}
    For both estimations the tools are the same. We use the water waves equations \eqref{WaterWavesEquations}, product estimates \ref{product estimate}, quotient estimates \ref{Quotient estimate}, and proposition \ref{Theorem 3.15}.\\
    Let's first prove the inequality on $\partial_t\zeta$. Denoting $s'=s+2$ and $\mathfrak{P}$ the Fourier multiplier defined as $\mathfrak{P}\colonequals \frac{|D|}{(1+\sqrt{\mu}|D|)^{1/2}}$ (I recall that $|D|$ means $|\xi|$ in Fourier space) we have
    \begin{align*}
        |\partial_t \zeta|_{H^{s'}} = |\frac{1}{\mu}\mathcal{G}^{\mu}[\epsilon\zeta]\psi|_{H^{s'}} \leq &N(s'+2).
    \end{align*}
    On the other hand for $\partial_t\nabla\psi$ we get
    \begin{align*}
        |\partial_t \psi|_{H^{s'}} \leq &|\zeta|_{H^{s'}} + ||\nabla\psi|^2|_{H^{s'}} +  |\frac{(\frac{1}{\sqrt{\mu}}\mathcal{G}^{\mu}\psi + \epsilon\sqrt{\mu}\nabla\zeta\cdot\nabla\psi)^2}{1+\epsilon^2\mu|\nabla\zeta|^2}|_{H^{s'}} \\
        \leq &|\zeta|_{H^{s'}} + |\nabla\psi|_{H^{s'}}^2 + C(\mu_{\max},\frac{1}{h_{\min}},||\nabla\zeta|^2|_{H^{s'}})|(\frac{1}{\sqrt{\mu}}\mathcal{G}^{\mu}\psi + \epsilon\sqrt{\mu}\nabla\zeta\cdot\nabla\psi)^2|_{H^{s'}}\\
        \leq &|\zeta|_{H^{s'}} + |\nabla\psi|_{H^{s'}}^2 + C(\mu_{\max},\frac{1}{h_{\min}},|\zeta|_{H^{s'+1}},|\frac{1}{\sqrt{\mu}}\mathcal{G}^{\mu}\psi|_{H^{s'}},
        |\nabla\zeta\cdot\nabla\psi|_{H^{s'}})\\
        \leq &N(s'+1).
    \end{align*}
    Thus
    \begin{align*}
        |\partial_t\nabla\psi|_{H^{s'}} \leq |\partial_t\psi|_{H^{s'+1}} \leq N(s'+2).
    \end{align*}
\end{proof}
\noindent Using the previous lemma, mixed with \eqref{Estimates on v_1} and \eqref{Estimates on partial t R} we get the control on $v_1$:
\begin{align*}
     ||\Lambda^s\nabla^{\mu}v_1||_2 \leq \mu^2\epsilon N(s+5).
\end{align*}

To get the control on $v_2$ we use a classical result for solutions of elliptic problems such as \eqref{Second elliptic problem for v2}. It is the lemma \ref{lemme 2.38}. Using also the fact that $g \colonequals \partial_t P(\Sigma_t)\nabla^{\mu}\widetilde{u}$ (see definition \ref{Definitions Fondamentales} for an expression of $P(\Sigma_t)$) it gives:  
\begin{align*}
    ||\Lambda^s\nabla^{\mu}v_2||_2 \leq &M(s+1)||\Lambda^s g||_2 \\
    ||\Lambda^s h\nabla^{\mu}\cdot P(\Sigma_t)\nabla^{\mu}\widetilde{\phi}_{\appr}||_2 \leq \mu^2\epsilon M(s+2)|\nabla\psi|_{H^{s+3}}.
\end{align*}
But $\widetilde{u}$ solves an elliptic problem for which we can use lemma \ref{lemma 3.43}. Using also lemma \ref{Estimates on partial derivatives of zeta and psi} and \eqref{Estimates on R} we get
\begin{align}\label{control on v2}
    ||\Lambda^s\nabla^{\mu}v_2||_2 \leq &C(M(s+1),|\partial_t\zeta|_{H^{s+3}}) \mu^2\epsilon M(s+1)||\Lambda^s R||_2 \leq N(s+5)
\end{align}
At the end, joining together the control on $v_1$ \eqref{Estimates on v_1} and the one on $v_2$ \eqref{control on v2} we get \eqref{control on u tilde step 2.1}. 

\noindent \underline{Step 2.2}: We can now give the control on \eqref{What we want to estimate 1}.
To do that we will first use \eqref{control on u tilde step 2.1} to prove the following lemma.
\begin{mylem}\label{Inequalities on partial t approximation of Vbar by nabla psi}
     Let $s \geq 0$, and $\zeta \in H^{s+6}(\mathbb{R}^d)$ be such that \eqref{FundamentalHypothesis} is satisfied. Let $\psi \in \dot{H}^{s+7}(\mathbb{R}^d)$, and $\overline{V}$ be as in \eqref{VerticallyAveragedHorizontalVelocity}. Let also $\mathrm{F}_1$ and $\mathrm{F}_2$ be the Fourier multipliers defined in proposition \ref{Proposition Approximation Vbar estimates}.\\
    The following estimates hold:
    \begin{align}\label{Lemme 4.6}
        \begin{cases}
            |\partial_t(\overline{V}-\mathrm{F}_1\nabla\psi)|_{H^s} \leq \mu\epsilon N(s+4),\\
             |\partial_t(\overline{V}-\nabla\psi-\frac{\mu}{3h}\nabla(h^3\mathrm{F}_2[\Delta\psi]))|_{H^s} \leq \mu^2 \epsilon N(s+6).
        \end{cases}
    \end{align}
\end{mylem}
\begin{proof} Let's first prove the second inequality. If we denote $\mu^2\epsilon R = \overline{V}-\widetilde{\overline{V}}_{\appr}$, having in mind $\overline{V}$ and $\widetilde{\overline{V}}_{\appr}$ written as in \eqref{Definition of Vbar and Vbartildeapp}, and the fact that through computations $\widetilde{\overline{V}}_{\appr} =\nabla\psi + \frac{1}{h}\nabla(h^3(\frac{\tanh{(\sqrt{\mu}|D|)}}{\sqrt{\mu}|D|} -1)\psi))$ (see \eqref{VbarTildeApp expression}), we have the following equality
\begin{align*}
    &|\partial_t(\overline{V}-\nabla\psi-\frac{1}{h}\nabla(h^3(\frac{\tanh{(\sqrt{\mu}|D|)}}{\sqrt{\mu}|D|} -1)\psi))|_{H^s} = \mu^2\epsilon |\partial_t R|_{H^s} \\
    = &|\int_{-1}^{0} (\nabla\partial_t\widetilde{u} - \frac{1}{h}(z\nabla h+\epsilon\nabla\zeta)\partial_z \partial_t \widetilde{u}) dz - \int_{-1}^{0} \partial_t(\frac{z\nabla h + \epsilon\nabla\zeta}{h})\partial_z \widetilde{u} dz|_{H^s}.\\
\end{align*}
So Poincaré's inequality (page 40 of \cite{WWP}) mixed up with product and quotient estimates \ref{product estimate}, \ref{Quotient estimate}, leads us to
\begin{align*}
    \mu^2\epsilon|\partial_t R|_{H^s} \leq ||\Lambda^{s+1}\nabla^{\mu}\partial_t\widetilde{u}||_2 \leq \mu^2 \epsilon N(s+6).  
\end{align*}
Moreover we defined $\mathrm{F}_2$ as
\begin{align*}
    (\frac{\tanh{(\sqrt{\mu}|D|)}}{\sqrt{\mu}|D|}-1)\psi = -\frac{\mu}{3}|D|^2 \mathrm{F}_2 \psi.
\end{align*}
Hence we come up with the estimation
\begin{align}\label{Estimation Partial t Vbar minus approximation nabla psi}
    |\partial_t(\overline{V}-\nabla\psi-\frac{\mu}{3h}\nabla(h^3\mathrm{F}_2[\Delta\psi]))|_{H^s} \leq \mu^2 \epsilon N(s+6).
\end{align}
    To prove the first inequality of \eqref{Lemme 4.6} we just need to do the same reasoning from step 2.1 to this point but taking instead $\widetilde{u} = \phi-\phi_0$, where $\phi_0$ is defined in \eqref{phi0}.
\end{proof}

\noindent Having in mind this proposition, we decompose \eqref{What we want to estimate 1} in two parts:
\begin{align*}
    &|\partial_t(\overline{V}-\nabla\psi-\frac{\mu}{3h}\nabla(h^3\mathrm{F}_2 \mathrm{F}_1^{-1}[\nabla\cdot\overline{V}]))|_{H^s} \\
    \leq & |\partial_t(\frac{\mu}{3h}\nabla(h^3\mathrm{F}_3[\nabla\cdot(\overline{V}-\mathrm{F}_1\nabla\psi)]))|_{H^s} + |\partial_t(\overline{V} - \nabla\psi -\frac{\mu}{3h}\nabla(h^3\mathrm{F}_2[\Delta\psi]))|_{H^s}.
\end{align*}
The bound on the second term is given by the second inequality of proposition \ref{Inequalities on partial t approximation of Vbar by nabla psi}.\\
The first term can be decomposed in three parts:
\begin{align*}
    &|\partial_t(\frac{\mu}{3h}\nabla(h^3\mathrm{F}_3[\nabla\cdot(\overline{V}-\mathrm{F}_1\nabla\psi)]))|_{H^s} \\
    \leq &|\frac{\mu\epsilon}{3h^2}\partial_t \zeta \nabla(h^3\mathrm{F}_3[\nabla\cdot(\overline{V}-\mathrm{F}_1\nabla\psi)])|_{H^s} + |\frac{\mu\epsilon}{3h}\nabla(3h^2\partial_t\zeta \mathrm{F}_3[\nabla\cdot(\overline{V}-\mathrm{F}_1\nabla\psi)])|_{H^s} \\
    + &|\frac{\mu}{3h}\nabla(h^3\mathrm{F}_3[\nabla\cdot\partial_t(\overline{V}-\mathrm{F}_1\nabla\psi)])|_{H^s}.
\end{align*}
Each of this terms are bounded using quotient estimates \ref{Quotient estimate}, product estimates \ref{product estimate}, proposition \ref{Vbar moins nabla psi estimation} and the first inequality of proposition \ref{Inequalities on partial t approximation of Vbar by nabla psi}.\\
At the end we get what we wanted to prove in this step 2: 
\begin{align*}
    |\partial_t(\overline{V}-\nabla\psi-\frac{\mu}{3h}\nabla(h^3\mathrm{F}_2 \mathrm{F}_1^{-1}[\nabla\cdot\overline{V}]))|_{H^s} \leq \mu^2 \epsilon N(s+6).
\end{align*}

\indent 

\noindent \underline{Step 3}: The last step is to bound 
\begin{align}\label{What we want to estimate 2}
    |\partial_t(\frac{\mu}{3h} \nabla(h^3\mathrm{F}_3[\nabla\cdot\overline{V}]) - \frac{\mu}{6h} \nabla(h^3\mathrm{F}_3[\nabla\cdot\overline{V}]+\mathrm{F}_3[h^3\nabla\cdot\overline{V}]))|_{H_s}
\end{align}
The main key is commutator estimates \ref{Commutator estimates}.
We decompose \eqref{What we want to estimate 2} into three parts:
\begin{align*}
     &|\partial_t(\frac{\mu}{3h} \nabla(h^3\mathrm{F}_3[\nabla\cdot\overline{V}]) - \frac{\mu}{6h} \nabla(h^3\mathrm{F}_3[\nabla\cdot\overline{V}]+\mathrm{F}_3[h^3\nabla\cdot\overline{V}]))|_{H_s} \\
     \leq &\mu\epsilon |\frac{\partial_t \zeta}{h^2}[\nabla(h^3-1)(\mathrm{F}_3-1)[\nabla\cdot\overline{V}]-(\mathrm{F}_3-1)[h^3-1\nabla\cdot\overline{V}])]|_{H^s} \\
     + &\mu\epsilon |\frac{1}{h}\nabla(\partial_t\zeta h^2 (\mathrm{F}_3-1)[\nabla\cdot\overline{V}]-(\mathrm{F}_3-1)[\partial_t\zeta h^2\nabla\cdot\overline{V}])|_{H^s} \\
     + &\mu\epsilon|\frac{1}{h}\nabla(\frac{h^3-1}{\epsilon}(\mathrm{F}_3-1)[\nabla\cdot\partial_t\overline{V}]-(\mathrm{F}_3-1)[\frac{h^3-1}{\epsilon}\nabla\cdot\partial_t \overline{V}])|_{H^s}\\
     \colonequals &T_1 + T_2 + T_3.
\end{align*}
Using quotient estimates \ref{Quotient estimate}, product estimates \ref{product estimate} and lemma \ref{Estimates on partial derivatives of zeta and psi} we get
\begin{align*}
    \begin{cases}
        T_1 \leq \mu\epsilon M(s+4)|[\mathrm{F}_3-1,h^3-1]\nabla\cdot\overline{V}|_{H^{s+1}},\\
        T_2 \leq \mu\epsilon M_0|[\mathrm{F}_3-1,\partial_t\zeta h^2]\nabla\cdot\overline{V}|_{H^{s+1}},\\
        T_3 \leq \mu\epsilon M_0|[\mathrm{F}_3-1,\frac{h^3-1}{\epsilon}]\nabla\cdot\partial_t\overline{V}|_{H^{s+1}}.
    \end{cases}
\end{align*}
But using the fact that for any $s\geq 0$ the operator $\Lambda^s$ and $\mathrm{F}_3-1$ commute we have
\begin{align*}
    \begin{cases}  
        |[\mathrm{F}_3-1,h^3-1]\nabla\cdot\overline{V}|_{H^{s+1}} &\leq |[(\mathrm{F}_3-1)\Lambda^{s+1},h^3-1]\nabla\cdot\overline{V}|_2 \\ &+|[\Lambda^{s+1},h^3-1](\mathrm{F}_3-1)[\nabla\cdot\overline{V}]|_2,\\
        |[\mathrm{F}_3-1,\partial_t\zeta h^2]\nabla\cdot\overline{V}|_{H^{s+1}} &\leq |[(\mathrm{F}_3-1)\Lambda^{s+1},\partial_t\zeta h^2]\nabla\cdot\overline{V}|_2 \\ &+|[\Lambda^{s+1},\partial_t\zeta h^2](\mathrm{F}_3-1)[\nabla\cdot\overline{V}]|_2,\\
        |[\mathrm{F}_3-1,\frac{h^3-1}{\epsilon}]\nabla\cdot\partial_t\overline{V}|_{H^{s+1}} &\leq |[(\mathrm{F}_3-1)\Lambda^{s+1},\frac{h^3-1}{\epsilon}]\nabla\cdot\partial_t\overline{V}|_2 \\
        &+|[\Lambda^{s+1},\frac{h^3-1}{\epsilon}](\mathrm{F}_3-1)[\nabla\cdot\partial_t\overline{V}]|_2.
    \end{cases}
\end{align*}
So using commutator estimates \ref{Commutator estimates} with $\mathrm{F}_3-1$ of order $2$ and $\mathcal{N}^2(\mathrm{F}_3-1) \lesssim \mu$ (see definition \ref{definition order of a Fourier multiplier} for the definition of $\mathcal{N}^2(\mathrm{F}_3-1)$), product estimates \ref{product estimate}, lemma \ref{Estimates on partial derivatives of zeta and psi}, proposition \ref{Vbar moins nabla psi estimation} and the first inequality in proposition \ref{Inequalities on partial t approximation of Vbar by nabla psi} we obtain
\begin{align*}
    \begin{cases}
        T_1 \leq \mu\epsilon M(s+4)(\mu|h^3-1|_{H^{s+3}}|\nabla\cdot\overline{V}|_{H^{s+2}} + |h^3-1|_{H^{\max{(t_0+1,s+1)}}}|(\mathrm{F}_3-1)[\nabla\cdot\overline{V}]|_{H^s}), \\
        T_2 \leq \mu\epsilon M(s+4)(\mu|\partial_t\zeta h^2|_{H^{s+3}}|\nabla\cdot\overline{V}|_{H^{s+2}} + |\partial_t\zeta h^2|_{H^{\max{(t_0+1,s+1)}}}|(\mathrm{F}_3-1)[\nabla\cdot\overline{V}]|_{H^s}),\\
        T_3 \leq \mu\epsilon M_0(\mu|\frac{h^3-1}{\epsilon}|_{H^{s+3}}|\nabla\cdot\partial_t\overline{V}|_{H^{s+2}} + |\nabla\frac{h^3-1}{\epsilon}|_{H^{\max{(t_0+1,s+1)}}}|(\mathrm{F}_3-1)[\nabla\cdot\partial_t\overline{V}]|_{H^s}).
    \end{cases}
\end{align*}
Here $t_0$ is a real number larger than $t_0$, see remark \ref{remark on t0}.

\noindent Hence using product estimates \ref{product estimate}, lemma \ref{Estimates on partial derivatives of zeta and psi}, proposition \ref{Vbar moins nabla psi estimation} and the first inequality in proposition \ref{Inequalities on partial t approximation of Vbar by nabla psi} we end up with
\begin{align*}
    \begin{cases}
        T_1 \leq \mu^2\epsilon M(s+5) |\nabla\psi|_{H^{s+5}} \leq N(s+5),\\
        T_2 \leq \mu^2\epsilon M(s+5) |\nabla\psi|_{H^{s+5}} \leq N(s+5),\\
        T_3 \leq \mu^2\epsilon N(s+7).
    \end{cases}
\end{align*}

It finishes the step 3 and the proof of proposition \ref{Consistency second full dispersion GN system section 4}, i.e. we proved the consistency of the water waves at order $O(\mu^2\epsilon)$ in the shallow water regime with the full dispersion Green-Naghdi system \eqref{SecondFullDispersionModel(InVbar)2} (with $n=6$).
\end{proof}

\begin{myrem}
    The $n=6$ regularity asked for deriving \eqref{SecondFullDispersionModel(InVbar)2} appeared only in the last step of the proof when we wanted to pass from a non-symmetric system to a symmetric one. Only $n=5$ is asked for the solutions of the water waves equations \eqref{WaterWavesEquations} to prove the consistency at order $O(\mu^2\epsilon)$ with respect to the system
    \begin{align}\label{Non symmetric system}
        \begin{cases}  
        \partial_t \zeta + \nabla \cdot(h\overline{V})=0, \\
        \partial_t(\overline{V}-\frac{\mu}{3h}\nabla(h^3\mathrm{F}_3[\nabla\cdot\overline{V}]) ) +\nabla \zeta &+\frac{\epsilon}{2} \nabla(|\overline{V}|^2) - \frac{\mu\epsilon}{3}\nabla(\frac{1}{h}\overline{V}\cdot \nabla(h^3\mathrm{F}_3[\nabla\cdot\overline{V}])) \\
        &- \frac{\mu\epsilon}{2}\nabla(h^2\mathrm{F}_3[\nabla\cdot\overline{V}]\nabla\cdot\overline{V}) =0.
        \end{cases}
    \end{align}
    However system \eqref{Non symmetric system} does not have a Hamiltonian formulation.
\end{myrem}

\label{Consistency with respect to the water waves equations}

\section{Consistency of other full dispersion models appearing in the literature}
\label{Consistency of the full dispersion models appearing in the literature}

\subsection{Full dispersion Green-Naghdi system}
In \eqref{approximation of grad psi in terms of V bar symmetrised version} we chose a kind of symmetrization which were naturally induced by an analogy with the one appearing in the first full dispersion Green-Naghdi system \eqref{first full dispersion model (in psi)} we derived in this paper. Another kind of symmetrization appears in the literature for a full dispersion Green-Naghdi system, see \cite{DucheneIsrawiTalhouk16}. They introduced a two-layer full dispersion Green-Naghdi system with surface tension in order to be able to study high-frequency Kevin-Helmholtz instabilities. In dimension $d=2$, without surface tension, their system for a one-layer fluid is
\begin{align}\label{Full dispersion Green Nagdhi model article Vincent}
     \begin{cases}  
        \partial_t \zeta + \nabla \cdot(h\overline{V})=0, \\
        \partial_t(\overline{V}-\frac{\mu}{3h}\nabla(\sqrt{\mathrm{F}_3}h^3\sqrt{\mathrm{F}_3}[\nabla\cdot\overline{V}]) ) +\nabla \zeta &+\frac{\epsilon}{2} \nabla(|\overline{V}|^2) - \frac{\mu\epsilon}{3}\nabla(\frac{1}{h}\overline{V}\cdot \nabla(\sqrt{\mathrm{F}_3}h^3\sqrt{\mathrm{F}_3}[\nabla\cdot\overline{V}])) \\
        &- \frac{\mu\epsilon}{2}\nabla(h^2\mathrm{F}_3[\nabla\cdot\overline{V}]\nabla\cdot\overline{V}) =0.
    \end{cases}
\end{align}
\begin{mypp}
    Let $\mathrm{F}_3$ be the Fourier multiplier defined in proposition \ref{Proposition Approximation Vbar estimates}. The water waves equations are consistent at order $O(\mu^2\epsilon)$ in the shallow water regime $\mathcal{A}$ with the system \eqref{Full dispersion Green Nagdhi model article Vincent}.
\end{mypp}
\begin{proof} I will only do a formal proof. The rigorous one would use the same tools as the proof of proposition \ref{Proposition consistency second Full dispersion Gn System} (see subsection \ref{Consistency with respect to the water waves equations}).

It is easy to see that
\begin{align}\label{Change of symmetrization}
    2\sqrt{\mathrm{F}_3}[h^3\sqrt{\mathrm{F}_3}[V]] = h^3\mathrm{F}_3[V] + \mathrm{F}_3[h^3 V] + O(\mu).
\end{align}
It only remains to use proposition \ref{Proposition consistency second Full dispersion Gn System} together with product estimates \ref{product estimate}, quotient estimates \ref{Quotient estimate} and the estimates on $\mathrm{F}_3$ of proposition \ref{FourierMultiplierApproximations} to get the result.
\end{proof}

The difference between \eqref{SecondFullDispersionModel(InVbar)} and \eqref{Full dispersion Green Nagdhi model article Vincent} in the mathematical point of view is of importance. Indeed the operator
\begin{align*}
    h(I_d - \frac{\mu}{3h}\nabla(\sqrt{\mathrm{F}_3}h^3\sqrt{\mathrm{F}_3}[\nabla\cdot\circ]))
\end{align*}
is invertible because one can decompose it in the following way:
\begin{align}\label{Symmetric decomposition}
    hI_d + \mu (\frac{h}{\sqrt{3}}\sqrt{\mathrm{F}_3}\nabla\cdot)^* h (\frac{h}{\sqrt{3}}\sqrt{\mathrm{F}_3}\nabla\cdot),
\end{align}
giving the positiveness of the operator and the coercivity of the bilinear form associated with. The Lax-Milgram theorem conclude \cite{DucheneIsrawiTalhouk16}.

However we don't have a similar decomposition as \eqref{Symmetric decomposition} for the operator $h(I_d - \frac{\mu}{6h}\nabla(h^3 \mathrm{F}_3[\nabla\cdot\circ] + \mathrm{F}_3[h^3\nabla\cdot\circ]))$. To ensure the invertibility it seems that we need an additionnal hypothesis on the smallness of $\epsilon\zeta$.

\label{Full dispersion Green-Naghdi system}

\subsection{Full dispersion Boussinesq systems}
In the literature several Whitham-Boussinesq systems (or full dispersion Boussinesq systems) are introduced, see \cite{Carter18, Dinvay19, Dinvay20, DinvayDutykhKalisch19, DinvayNilsson19, DinvaySelbergTesfahun19, HurPandey16, NilssonWang19, VargasMaganaPanayotaros16}. We pay a particular attention to the one studied in \cite{DinvaySelbergTesfahun19} for which they proved a local well-posedness result in dimension $d=2$ and a global well-posedness result for small data in dimension $d=1$. We recall it for $d=2$
\begin{align}\label{Full Dispersion Boussinesq system Dinvay}
    \begin{cases}
        \partial_t \zeta + \mathrm{F}_1\Delta\psi + \epsilon \mathrm{F}_1\nabla\cdot(\zeta \mathrm{F}_1\nabla\psi)=0,\\
        \partial_t\nabla\psi + \nabla\zeta + \frac{\epsilon}{2}\nabla(\mathrm{F}_1|\nabla\psi|)^2 = 0.
    \end{cases}
\end{align}
\begin{mypp}\label{Consistency Full Dispersion Dinvay system}
The water waves equations \eqref{WaterWavesEquations} are consistent at order $O(\mu\epsilon)$ in the shallow water regime $\mathcal{A}$ with the following system
\begin{align}
    \begin{cases}
        \partial_t \zeta + \mathrm{F}_1\Delta\psi + \epsilon \mathrm{F}_1\nabla\cdot(\zeta \mathrm{F}_1\nabla\psi)=0,\\
        \partial_t\psi + \zeta + \frac{\epsilon}{2}(\mathrm{F}_1|\nabla\psi|)^2 = 0.
    \end{cases}
\end{align}
\end{mypp}
\begin{proof}
Again I will only do a formal proof. The rigorous one would use the same tools as the one of proposition \ref{Proposition consistency first full dispersion GN system}.

To do so let's use the fact that we proved the consistency of the water waves equations at order $O(\mu^2\epsilon)$ with system \eqref{FirstFullDispersionModel(InPsi)}, we discard all the terms of order $O(\mu\epsilon)$ of the latter. We obtain a formal consistency of the water waves system at order $O(\mu\epsilon)$ with the system
\begin{align*}
    \begin{cases} 
        \partial_t \zeta +\mathrm{F}_1\Delta\psi + \epsilon\nabla\cdot(\zeta\nabla\psi) = 0, \\
        \partial_t \psi  +\zeta + \frac{\epsilon}{2}|\nabla\psi|^2 = 0,
    \end{cases}
\end{align*}
(We used the identity $\Delta \psi + \frac{\mu}{3} \Delta \mathrm{F}_2 \Delta \psi = \mathrm{F}_1 \Delta \psi$).\\
It only remains to have in mind product estimates \ref{product estimate}, and the estimates on $\mathrm{F}_1$ of proposition \ref{FourierMultiplierApproximations} to see that taking $(\zeta,\psi)$ solutions of the water waves system \eqref{WaterWavesEquations}, one has 
\begin{align*}
    \begin{cases}
        \partial_t \zeta + \mathrm{F}_1\Delta\psi + \epsilon \mathrm{F}_1\nabla\cdot(\zeta \mathrm{F}_1\nabla\psi)=O(\mu\epsilon),\\
        \partial_t\psi + \zeta + \frac{\epsilon}{2}(\mathrm{F}_1|\nabla\psi|)^2 = O(\mu\epsilon).
    \end{cases}
\end{align*}
This conclude the formal demonstration.
\end{proof}

\begin{myrem}
    \begin{itemize}
        \item We could easily adapt definition \ref{Definition consistency} to match with system \eqref{Full Dispersion Boussinesq system Dinvay}. And say that the water waves equations \eqref{WaterWavesEquations} are consistent with this Whitham-Boussinesq system at order $O(\mu\epsilon)$.
        \item Using the same tools we could prove the consistency at order $O(\mu\epsilon)$ of the water waves equations \eqref{WaterWavesEquations} with the other Whitham-Boussinesq systems of the literature. 
        \item As the proof of proposition \ref{Consistency Full Dispersion Dinvay system} makes clear, the water waves equations are consistent at order $O(\mu\epsilon)$ with every systems 
         \begin{align*}
            \begin{cases}
                \partial_t\zeta + \mathrm{F_1}\Delta\psi + \epsilon\mathrm{G}_1\nabla\cdot(\zeta\mathrm{G}_2\nabla\psi) = 0,\\
                \partial_t\nabla\psi + \nabla\zeta + \frac{\epsilon}{2}\mathrm{G}_3\nabla(|\mathrm{G}_4\nabla\psi|^2) = 0.
            \end{cases}
        \end{align*}
         where the Fourier multipliers $\mathrm{G}_1$, $\mathrm{G}_2$, $\mathrm{G}_3$ and $\mathrm{G}_4$ are approximations of identity of the type $\mathrm{G}_i = 1 + O(\mu)$. However, the well-posedness properties of the system will depend on the characteristics of the Fourier multipliers and in particular the order of their symbol (definition \ref{definition order of a Fourier multiplier}). We postpone the study of the well-posedness of such systems to a future work.
    \end{itemize}
\end{myrem}

\label{Full dispersion Boussinesq systems}

\appendix
\section{Technical tools}
\begin{mypp}\label{product estimate}(Product estimates)
    \begin{enumerate}
        \item Let $t_0 > d/2$, $s\geq-t_0$ and $f \in H^s\cap H^{t_0}(\mathbb{R}^d), g\in H^s(\mathbb{R}^d)$. Then $fg \in H^s(\mathbb{R}^d)$ and 
        \begin{align*}
            |fg|_{H^s} \lesssim |f|_{H^{\max{(t_0,s)}}}|g|_{H^s}
        \end{align*}
        \item Let $s_1, s_2 \in \mathbb{R}$ be such that $s_1+s_2 \geq 0$. Then for all $s\leq s_j$ $(j = 1,2)$ and $s < s_1 + s_2 - d/2$, and all $f\in H^{s_1}(\mathbb{R}^d), g \in H^{s_2}(\mathbb{R}^d)$, one has $fg \in H^s(\mathbb{R}^d)$ and 
        \begin{align*}
            |fg|_{H^s} \lesssim |f|_{H^{s_1}}|g|_{H^{s_2}}
        \end{align*}
    \end{enumerate}
\end{mypp}
\begin{proof}
See Appendix B.1 in \cite{WWP}.
\end{proof}

\begin{mypp}\label{Quotient estimate}(Quotient estimates)
    Let $t_0>d/2, s\geq-t_0$ and $c_0>0$. Also let $f \in H^s(\mathbb{R}^d)$ and $g\in H^s\cap H^{t_0}(\mathbb{R}^d)$ be such that for all $X \in \mathbb{R}^d$, one has $1 + g(X) \geq c_0$. Then $\frac{f}{1+g}$ belongs to $H^s(\mathbb{R}^d)$ and 
    \begin{align*}
        |\frac{f}{1+g}|_{H^s} \leq C(\frac{1}{c_0},|g|_{H^{\max{(t_0,s)}}})|f|_{H^s}
    \end{align*}
\end{mypp}
\begin{proof}
    See Appendix B.1 in \cite{WWP}.
\end{proof}

\begin{mylem}\label{lemma 3.43}
    Let $P(\Sigma_t)$ be as in definition \ref{Definitions Fondamentales}. Let $h \in L_z^2H^s_X((-1,0)\times\mathbb{R}^d)$ and $u \in L_z^2H^{s+1}_X \cap H^1_z H^s_X((-1,0)\times\mathbb{R}^d)$ $(s\geq 0)$ solve the boundary value problem
    \begin{align*}
        \begin{cases}
            \nabla^{\mu}\cdot P(\Sigma_t) \nabla^{\mu} u = h,\\
            u|_{z=0} = 0, \ \ \partial_z u|_{z=-1} = 0.
        \end{cases}
    \end{align*}
    Then one has 
    \begin{align*}
        ||\Lambda^s\nabla^{\mu} u||_2 \leq M(s+1)||\Lambda^s h||_2.
    \end{align*}
\end{mylem}
\begin{proof}
    See lemma 3.43 in \cite{WWP}.
\end{proof}

\begin{mypp}\label{FourierMultiplierApproximations}
Let $s\geq 0$, $z\in(-1,0)$ and $\psi$ such that $\nabla\psi \in H^{s+1}(\mathbb{R}^d)$, then we have the following estimations
\begin{align*}
     \begin{cases}
        |(\frac{1-\mathrm{F}_0}{\mu|D|^2} +\frac{z^2}{2} +z)\psi|_{H^s}\lesssim \mu|\nabla\psi|_{H^{s+1}} \ \ , \ \ |(1-(z+1)^2 \mathrm{F}_0 +z^2+2z)\psi|_{H^s} \lesssim \mu|\nabla\psi|_{H^{s+1}} \\
        |(\frac{\tanh(\sqrt{\mu}|D|)}{\sqrt{\mu}|D|} -1)\psi|_{H^s} \lesssim \mu|\nabla\psi|_{H^{s+1}} \ \ , \ \ |(\frac{z+1}{\sqrt{\mu}|D|} \frac{\sinh((z+1)\sqrt{\mu}|D|)}{\cosh(\sqrt{\mu}|D|)} - (z+1)^2)\psi|_{H^s} \lesssim \mu|\nabla\psi|_{H^{s+1}} \\
        |(\mathrm{F}_2 -1)\psi|_{H^s} \lesssim \mu |\nabla\psi|_{H^{s+1}} \ \ , \ \ |(\mathrm{F}_3-1)\psi|_{H^s} \lesssim \mu |\nabla\psi|_{H^{s+1}}
    \end{cases}
\end{align*}
An estimation of order $O(\mu^2)$ for $\frac{\tanh(\sqrt{\mu}|D|)}{\sqrt{\mu}|D|}$ will also be useful. If $\nabla\psi \in H^{s+3}$ then
\begin{align*}
    |(\frac{\tanh(\sqrt{\mu}|D|)}{\sqrt{\mu}|D|} -1 + \frac{1}{3}\mu|D|^2)\psi|_{H^s} \lesssim \mu^2|\nabla\psi|_{H^{s+3}}
\end{align*}
\end{mypp}
\begin{proof}
All the proves are similar and the main key is the Taylor-Lagrange formula. All these estimations except the last one are on the same form where $G$ is a smooth function on $(0,+\infty)$, continuous in $0$.
\begin{align*}
    |(G(\sqrt{\mu}|D|)-G(0))\psi|_{H^s} \leq \mu |\nabla\psi|_{H^{s+1}} \iff |(G(\sqrt{\mu}|\xi|)-G(0))\hat{\psi}|_{H^s} \leq \mu |\nabla\psi|_{H^{s+1}}
\end{align*}
For the last one it would be 
\begin{align*}
    &|(G(\sqrt{\mu}|D|)-G(0)-\mu|D|^2 G''(0))\psi|_{H^s} \leq \mu^2 |\nabla\psi|_{H^{s+3}} \\
    \iff &|(G(\sqrt{\mu}|\xi|)-G(0)-\mu|\xi|^2 G''(0))\hat{\psi}|_2 \leq \mu^2 |\nabla\psi|_{H^{s+3}}
\end{align*}
If we succeed in proving that the second derivative of $G$ is bounded in $[0,+\infty)$ , and that $G'(0)=0$ then we can use the Taylor-Lagrange formula stating that for all $x\in [0,+\infty)$ there exists $\theta \in [0,1]$ such that
\begin{align*}
    G(x)-G(0) = \frac{x^2}{2}G''(\theta x)
\end{align*}
then the boundedness of $G''$ allow us to write
\begin{align*}
    |G(x)-G(0)| \leq |G''|_{\infty}x^2
\end{align*}
Replacing $x$ by $\sqrt{\mu}|\xi|$ in the last inequality we obtain
\begin{align*}
    |(G(\sqrt{\mu}|\xi|)-G(0))\hat{\psi}|_2 \leq \mu||\xi|^2\hat{\psi}|_2 \leq \mu|\nabla\psi|_{H^1}
\end{align*}
Thus it is sufficient to prove the boundedness in $C^2([0,+\infty))$ of $G$ and the fact that $G'(0)=0$ for the following functions:
\begin{align*}
    \begin{cases}
        G_1(x) = \frac{1}{x^2}(1-\frac{\cosh{((z+1)x)}}{\cosh{(x)}}), \ \ G_2(x) =  1-(z+1)^2\frac{\cosh{((z+1)x)}}{\cosh{(x)}} \\
        G_3(x) = \frac{\tanh{(x)}}{x}, \ \ G_4(x) = \frac{z+1}{x}\frac{\sinh{((z+1)x)}}{\cosh{(x)}} \\
        G_5(x) = \frac{3}{x\tanh{(x)}} - \frac{3}{x^2}, \ \ G_6(x) = \frac{3}{\mu}(1-\frac{\tanh{(x)}}{x})
    \end{cases}
\end{align*}
The end of the proof is let to the reader.
\end{proof}

\begin{mydef}\label{definition order of a Fourier multiplier}
We say that a Fourier multiplier $\mathrm{F}(D)$ is of order $s$ $(s\in\mathbb{R})$ and write $\mathrm{F} \in \mathcal{S}^s$ if $\xi \in \mathbb{R}^d \mapsto F(\xi) \in \mathbb{C}$ is smooth and satisfies
\begin{align*}
    \forall \xi \in \mathbb{R}^d, \forall\beta\in\mathbb{N}^d, \ \ \sup_{\xi\in\mathbb{R}^d} \langle\xi\rangle^{|\beta|-s}|\partial^{\beta}F(\xi)| < \infty.
\end{align*}
We also introduce the seminorm
\begin{align*}
    \mathcal{N}^s(F) = \sup_{\beta\in\mathbb{N}^d,|\beta\leq2+d+\lceil\frac{d}{2}\rceil} \sup_{\xi\in\mathbb{R}^d} \langle\xi\rangle^{|\beta|-s}|\partial^{\beta}F(\xi)|.
\end{align*}
\end{mydef}

\begin{mypp}\label{Commutator estimates}
    Let $t_0 > d/2$, $s \geq 0$ and $F \in \mathcal{S}^s$. If $f\in H^s\cap H^{t_0+1}$ then, for all $g \in H^{s-1}$,
    \begin{align*}
        |[F(D),f]g|_2 \leq \mathcal{N}^s(F)|f|_{H^{\max{(t_0+1,s)}}}|g|_{H^{s-1}}. 
    \end{align*}
\end{mypp}
\begin{proof}
    See Appendix B.2 in \cite{WWP} for a proof of this proposition.
\end{proof}

\begin{mypp}\label{Theorem 3.15}
    Let $s\geq 2$. Let $\zeta \in H^{s+2}(\mathbb{R}^d)$ be such that \eqref{FundamentalHypothesis} is satisfied and $\psi \in \dot{H}^{s+1}(\mathbb{R}^d)$.
    Then one has 
    \begin{align*}
        \begin{cases}
            |\frac{1}{\mu}\mathcal{G}^{\mu}\psi|_{H^{s}} \leq M(s+2)|\nabla\psi|_{H^{s+1}},\\
            |\frac{1}{\sqrt{\mu}}\mathcal{G}^{\mu}\psi|_{H^{s}} \leq \mu^{1/4}M(s+1)|\nabla\psi|_{H^s}.
        \end{cases}
    \end{align*}
\end{mypp}
\begin{proof}
    This is a direct consequence of theorem 3.15 in \cite{WWP}.
\end{proof}

\begin{mylem}\label{lemme 2.38}
    Let $t_0 > d/2$, $s \geq 0$. Let $P(\Sigma_t)$ be as in definition \ref{Definitions Fondamentales}. Let $g(X,z)$ be a function on $\mathcal{S}\colonequals \mathbb{R}^d\times (-1,0)$ sufficiently regular such that its trace at $z=-1$ makes sense. Let $u$ solve the boundary value problem
    \begin{align*}
    \begin{cases}
        \nabla^{\mu}\cdot P(\Sigma_t) \nabla^{\mu} u = -\nabla^{\mu}\cdot g, \\
        u|_{z=0}=0, \ \ v_2|_{z=-1}=-e_z\cdot g|_{z=-1}.
    \end{cases}
    \end{align*}
    Then one has
    \begin{align*}
        ||\Lambda^s\nabla^{\mu}u||_2 \leq M(s+1)||\Lambda^s g||_2. 
    \end{align*}
\end{mylem}
\begin{proof}
    See lemma 2.38 in \cite{WWP}.
\end{proof}

\bibliographystyle{plain}
\bibliography{Biblio.bib}

\end{document}